\documentclass[a4paper,12pt,final]{article}
\usepackage[english]{babel}
\usepackage[utf8]{inputenc}
\usepackage[T1]{fontenc}
\usepackage{lmodern}
\usepackage{mathtools}
\usepackage{amssymb,amsmath,amsthm,mathrsfs}
\usepackage{geometry}
\usepackage{indentfirst}
\usepackage[pdfauthor={Author's name},%
pdftitle={Document Title},%
pagebackref=true,%
pdftex]{hyperref}
\usepackage[square,numbers]{natbib}

\makeatletter
\numberwithin{equation}{section}
\numberwithin{figure}{section}
\theoremstyle{plain}
\newtheorem{thm}{\protect\theoremname}
  \theoremstyle{plain}
  \newtheorem{cor}[thm]{\protect\corollaryname}
  \theoremstyle{plain}
  \newtheorem{lem}[thm]{\protect\lemmaname}
  \theoremstyle{plain}
  \newtheorem{prop}[thm]{\protect\propositionname}
  \theoremstyle{definition}
  \newtheorem{defn}[thm]{\protect\definitionname}
  \theoremstyle{plain}
  
  \theoremstyle{remark}
  \newtheorem{rem}[thm]{\protect\remarkname}
\newtheorem{exam}[thm]{\protect\examplename}
\newtheorem{notat}[thm]{\protect\notationname}
\makeatother

\providecommand{\conjecturename}{Conjecture}
  \providecommand{\corollaryname}{Corollary}
  \providecommand{\definitionname}{Definition}
  \providecommand{\lemmaname}{Lemma}
  \providecommand{\propositionname}{Proposition}
  \providecommand{\remarkname}{Remark}
\providecommand{\examplename}{Example}
\providecommand{\theoremname}{Theorem}
\providecommand{\notationname}{Notation}

\makeatletter

\makeatletter

\DeclareMathOperator{\inv}{Inv}
\DeclareMathOperator{\cl}{cl}
\DeclareMathOperator{\id}{id}
\DeclareMathOperator{\dd}{d}
\DeclareMathOperator{\Ima}{Im}
\DeclareMathOperator{\supp}{supp}
\DeclareMathOperator{\diam}{diam}
\makeatother

\title{A Discontinuous Differential Calculus in the Framework Colombeau's
Full Algebra}

\author{Wagner Cortes\thanks{Instituto de Matemática, Universidade
   Federal do Rio Grande do Sul, Porto Alegre-RS, Brazil, Av. Bento
   Gonçalves, 9500, 91509-900, E-mail: {\it{wcortes@gmail.com}}.} \and 
Antonio R. G. Garcia\thanks{Centro de Ciências Exatas e Naturais,
 Universidade Federal Rural do Semi-Árirido, Mossoró-RN, Brazil,
  Av. Francisco Mota, 572, 59625-900, E-mail: {\it{ronaldogarcia@ufersa.edu.br}}.} 
  \and 
Severino H. da Silva\thanks{Unidade Acadêmica de Matemática,
 Universidade Federal de Campina Grande, Campina Grande-PB, Brazil,
  Av. Aprígio Veloso, 785, 58429-970, E-mail: {\it{horacio@mat.ufcg.edu.br}}.}}

\begin{document}
\maketitle

\begin{abstract}
Starting from the Colombeau's full generalized functions, the sharp
topologies and the notion of generalized points, we introduce a new
kind differential calculus (for functions between totally disconnected
spaces). We study generalized pointvalues, Colombeau's differential
algebra, holomorphic and analytic functions. We show that   
the Embedding Theorem and the Open Mapping Theorem hold 
in this framework. Moreover, we study some applications in
differential equations.  
\end{abstract}



\section{Introduction}\label{sec-1}
The theory of Colombeau generalized functions appears in the 
1980's which the main aim was to define a product on the
distribution space, see \cite{JFC3,JFC1,JFC2,JFC4} and \cite{MK} 
for more details. Colombeau constructed differential algebras
$\mathcal{G}(\Omega),~\Omega\subseteq\mathbb{R}^n$ containing the
space of smooth functions defined on $\Omega$ as a subalgebra and the
space of distributions on $\Omega$ as a subspace, i.e, it was constructed an 
associative, commutative differential algebra containing the space of
distributions and hold the Leibniz rule for differential product of
two distributions, where nowadays this algebra is known as Colombeau's 
algebras. So far, Colombeau algebras are the only known differential
algebras having all these properties enumerated above.

Fundamental investigations about the structure of these algebras 
containing the distributions space have been carried out
by Rosinger, see \cite{EER1,EER2,EER3} and \cite{EER4} for more
details. Moreover, with the results of \cite{EER1,EER2,EER3} and \cite{EER4}
were constructed a general theory which characterized algebras of generalized
functions containing the space of distributions, but was Colombeau
that constructed differential algebras with good properties (see
\cite{JFC1} and \cite{JFC2}). The Colombeau algebras, simplified and
full, rapidly developed during the last years and was found useful 
applications to linear and nonlinear partial differential equations, 
calculus of variations, mathematical physics problem, stochastic
analysis and in differential geometry, where the theory of distributions have
limitations of applications, because it is not nonlinear. In the presence
of nonlinearity and the nonlinear generalized functions (Colombeau
algebras), it produces new insight where the classical theory does not.  

Our idea is to extend some results of \cite{OJR}, where the authors 
developed the discontinuous calculus for Colombeaus's simplified generalized functions.  
The definitions and topology of Colombeau's full generalized numbers 
are far more complicated than the simplified version which makes an interesting
object to introduce the differential calculus. It is important to say
that Colombeau algebras is not just a regularization of functions,
but an extension of the classical functions with the differential
calculus developed here.

The starting point of an algebraic theory of the topological ring of 
Colombeau's full generalized numbers was made in \cite{JRJ}. 
This algebraic theory, together with the theory of the point values 
of a Colombeau's full generalized functions that will be developed
here, which is based with the constructions 
made in \cite{MK}, are the base to introduce the differential
calculus in the framework of Colombeau's full generalized functions.

This paper is the continuation of a program whose aim is the
development of a differential calculus in the Colombeau algebras
setting. This program was effectively started in \cite{OJR} and
\cite{OJRO}, where the main point was to develope the differential 
calculus. Thus, the Colombeau's simplified theory is a natural
extension of classical calculus. We want to extend these studies for 
the Colombeau's full theory, we want to show that the differential
calculus in the framework of Colombeau's full theory is a natural 
extension of classical calculus, too. As a future work, we want to  
study integration of generalized functions over membranes in the
context of Colombeau's full generalized numbers
where it extends the ideas presented in \cite{OJRO}. 

This work is organized as follows: In the Section
\ref{sec-init}, we collect some basic definitions, results
and notations to be used in the sequel of the paper and, as a rule,
most of the proofs are omitted. In the Section \ref{sec-2}, we present
some results that are extensions of results obtained in \cite{OJR},
observing that the  results obtained in \cite{JRJ} was assumed that the
support of the mollifiers are contained in the ball of center in $x_0$
and radio $1$. In fact, we consider $\diam(\supp(\varphi))=1$ 
for $\varphi\in\mathcal{A}_0(\mathbb{K})$. In this case,
we have that its support number, $\dd(\varphi)=\sup\{|x|:\varphi(x)\ne
0\}=\varepsilon$, see \cite{NPS} for more details. In the Section 
\ref{sec-3}, we introduce the pointvalues  
in the framework Colombeau's full algebras and we  extend  
some results in \cite{MK}. In the Section \ref{sec-4}, we study the Colombeau
differential full algebra over the image of $\mathcal{G}(\Omega)$ 
by the operator $\kappa$ that we  define in  this section. In the Section  
\ref{sec-5}, we study the holomorphic and analytic generalized
function and some applications in the framework Colombeau's full
generalized functions. 

\section{Definitions, results and notations}\label{sec-init}

In this section we recall some basic definitions, results and notations that will be necessary 
to  the development of this work. As a rule, the proofs will be  ommited.

\begin{notat}\label{nota}
\begin{enumerate}
\item[$a)$] $I\coloneqq ]0,1],~\bar{I}\coloneqq [0,1]$ and
  $I_\eta\coloneqq ]0,\eta[,~\forall~\eta\in I$.
\item[$b)$] $A\setminus B\coloneqq\{a\in A:a\notin B\}$.
\item[$c)$] $\mathbb{Q}$ denotes the field of rational numbers.
\item[$d$)] $\mathbb{K}$ denotes either the field of real or complex
  numbers, i.e., $\mathbb{R}$ or $\mathbb{C}$.
\item[$e)$] $\mathbb{K}^*\coloneqq\mathbb{K}\setminus\{0\}$.
\item[$f)$] $\mathbb{N}$ and $\mathbb{Z}$ stand respectively for the
  set natural numbers and the set of
  integers. Moreover, $\mathbb{N}^*\coloneqq\mathbb{N}\setminus\{0\}$ and
  $\mathbb{Z}^*\coloneqq\mathbb{Z}\setminus\{0\}$.
\item[$g)$] $\mathbb{R}_+\coloneqq\{x\in\mathbb{R}:x\ge 0\}$ and
  $\mathbb{R}_+^*\coloneqq\{x\in\mathbb{R}:x>0\}$.
\item[$h)$] We denote  $\overline{\mathbb{K}}_s$ as the topological ring of
  Colombeau's simplified generalized numbers, see \cite{JO}.
\item[$i)$]
  $\mathcal{A}_0(\mathbb{R})\coloneqq\{\varphi\in\mathcal{D}(\mathbb{R}):\intop_0^\infty\varphi(x)
  \dd
  x=\frac{1}{2},~\varphi ~\mbox{is even and}~\varphi\equiv\mbox{const.
    in}~V_0\}$, where $V_0$ is a neighborhood of the origin. If
  $\varphi\in\mathcal{A}_0(\mathbb{R})$, then its support number is
  $\dd(\varphi)\coloneqq\sup\{|x|:\varphi(x)\ne 0\}$, see \cite{NPS}.
\item[$j)$]
  $\mathcal{A}_q(\mathbb{R})\coloneqq\{\varphi\in\mathcal{A}_0(\mathbb{R}):\intop_0^\infty
  x^{\frac{j}{m}}\varphi(x)\dd x=0,~\mbox{for}~1\le j,m\le
  q,~q\in\mathbb{N}\}$. If
  $\varphi\in\mathcal{A}_q(\mathbb{R}),~q\in\mathbb{N}$, then for
  every
  $\varepsilon>0,~\varphi_\varepsilon(x)=\varepsilon^{-n}\varphi(\frac{x}{\varepsilon}),~x\in\mathbb{R}^n$,
  belongs to $\mathcal{A}_q(\mathbb{R})$, see \cite{NPS}. 
\item[$l$)]
  $\Gamma\coloneqq\{\gamma:\mathbb{N}\to\mathbb{R}_+:\gamma(n)<\gamma(n+1),~
\forall~n\in\mathbb{N}~\mbox{and}~\lim\limits_{n\to\infty}\gamma(n)=\infty\}$
  is the set of the strict increasing sequences diverging to infinity
  when $n\to\infty$.
\item[$m)$]  $\overline{\mathbb{K}}$
  denotes the topological ring of Colombeau's full generalized
  numbers, see (\cite{JRJ}, Definition 1.2) 
\item[$n)$] For each $\Omega\subseteq\mathbb{R}^n$, denote
  $\mathcal{G}_s(\Omega)$ as the topological ring of Colombeau's
  simplified generalized functions, see \cite{JO}.
\end{enumerate}
\end{notat}

Let $\mathcal{E}(\Omega)$ be the ring (pointwise operations) of the
functions $u:\mathcal{A}_0(\mathbb{K})\times\Omega\to\mathbb{K}$ such
that $u(\varphi,\cdot)=u_\varphi(\cdot)\in\mathscr{C}^\infty(\Omega)$
for each $\varphi\in\mathcal{A}_0(\mathbb{K})$. If
$\alpha\in\mathbb{N}^n$ and $x\in\Omega$ we set $\partial^\alpha
u_\varphi(x)\coloneqq\partial^\alpha u_\varphi(\cdot)(x)$. Let
$\mathcal{E}_M(\Omega)$ be  the subring of $\mathcal{E}(\Omega)$
consisting of those functions satisfying the following ``moderation''
condition:

$M)$ $\forall~\alpha\in\mathbb{N}^n~\exists~N\in\mathbb{N}$ such that
$\forall~\varphi\in\mathcal{A}_N(\mathbb{K})~\exists~
c=c(\varphi)>0$ and $\eta=\eta(\varphi)\in I$
verifying \[\Vert\partial^\alpha
u_{\varphi_\varepsilon}(\cdot)\Vert\le
c\varepsilon^{-N},~\forall~\varepsilon\in I_\eta.\]
 We define an ideal  $\mathcal{N}(\Omega)$ of $\mathcal{E}_M(\Omega)$
 as the set of  $u\in \mathcal{E}_M(\Omega)$  
that satisfies the following ``nulity'' condition: 

$N)$ $\forall~\alpha\in\mathbb{N}^n~\exists~N\in\mathbb{N}$ and
$\gamma\in\Gamma$ such that $\forall~q\ge N$ and
$\forall~\varphi\in\mathcal{A}_q(\mathbb{K})~\exists~c=c(\varphi)>0$
and $\eta=\eta(\varphi)\in I$ verifying \[\Vert\partial^\alpha
u_{\varphi_\varepsilon}(\cdot)\Vert\le
c\varepsilon^{\gamma(q)-N},~\forall~\varepsilon\in I_\eta.\]

Note that $\mathcal{N}(\Omega)$ is a maximal differential ideal of $\mathcal{E}_M(\Omega)$.
The Colombeau's full generalized functions on $\Omega$ is defined
by \[\mathcal{G}(\Omega)\coloneqq\mathcal{E}_M(\Omega)/\mathcal{N}(\Omega).\]

This definition appears  in \cite{AGJ} for example.  Now, in the case of 
$\bar{\Omega}$, the topological sheaf of
$\Omega\subset\mathbb{R}^n$  we have the Colombeau's full generalized
functions on $\bar{\Omega},~\mathcal{G}(\bar{\Omega})=
\mathcal{E}_M(\bar{\Omega})/\mathcal{N}(\bar{\Omega})$.

Note that if $\varphi\in\mathcal{D}(\Omega), ~\intop_{\Omega}\varphi(x)\dd x=1$
and $\supp(\varphi)\subseteq \bar{B}_1(0)$. Then, for all
$\varepsilon>0,~\varphi_\varepsilon(x)=\varepsilon^{-n}\varphi(\frac{x}{\varepsilon}
),~x\in\Omega\subseteq\mathbb{R}^n$ is such that
$\varphi_\varepsilon\in\mathcal{D}(\Omega),~\intop_{\Omega}\varphi_\varepsilon(x)\dd
x=1$ and $\supp(\varphi_\varepsilon)\subseteq\bar{B}_\varepsilon(0)$.  

In \cite{JO}, was defined an interesting subgroup of
$\inv(\overline{\mathbb{K}}_s)$, i.e., \[Q\coloneqq\{\alpha_r:r\in\mathbb{R}\},\]
where $\alpha_r:I\to\mathbb{R}_+^*$ defined by
$\alpha_r(\varepsilon)=\varepsilon^r$ with inverse
$\alpha_{-r}:I\to\mathbb{R}_+^*$ given by
$\alpha_{-r}(\varepsilon)=\varepsilon^{-r}$. In \cite{JRJ}, was
defined its correpondent subgroup of
$\inv(\overline{\mathbb{K}})$, \[H\coloneqq\{\dot{\alpha}_r:r\in\mathbb{R}\},\]
where $\dot{\alpha}_r:\mathcal{A}_0(\mathbb{K})\to\mathbb{R}_+^*$ is given
by $\dot{\alpha}_r(\varphi)=(i(\varphi))^r$ ($i(\varphi)>0$ that is the
diameter of the $\supp(\varphi),~\varphi\in\mathcal{A}_0(\mathbb{K})$) with inverse
$\dot{\alpha}_{-r}:\mathcal{A}_0(\mathbb{K})\to\mathbb{R}_+^*$ 
given by $\dot{\alpha}_{-r}(\varphi)=(i(\varphi))^{-r}$. In
particular, \[i(\varphi_\varepsilon)=\varepsilon
i(\varphi),~\forall~\varepsilon>0\qquad\mbox{and}\qquad
\dot{\alpha}_{r}(\varphi_\varepsilon)=\varepsilon^r(i(\varphi))^r=\alpha_r(\varepsilon)\dot{\alpha}_r(\varphi).\] 
Hence, if $i(\varphi)\le 1$, then
$\dot{\alpha}_r(\varphi_\varepsilon)\le
\varepsilon^r=\alpha_r(\varepsilon),~\forall~r\in\mathbb{R}$ and if
$i(\varphi)=1$, then
$\dot{\alpha}_{r}(\varphi_\varepsilon)=\alpha_r(\varepsilon)=\varepsilon^r$,
i.e., \[H_{\varphi_\varepsilon}\coloneqq\{\dot{\alpha}_r(\varphi_\varepsilon)|r\in\mathbb{R}\}\subseteq
Q_\varepsilon\coloneqq\{\alpha_r(\varepsilon):r\in\mathbb{R}\}
\qquad\mbox{and}\qquad
H_{\varphi_\varepsilon}=Q_\varepsilon=\{\varepsilon^r:r\in\mathbb{R}\}\]
if and only if $i(\varphi)=1$ for
$\varphi\in\mathcal{A}_0(\mathbb{K})$.

If $\varphi\in\mathcal{A}_0(\mathbb{K})$, then its support number
$\dd(\varphi)$ is defined as in the Notation \ref{nota}, item $i)$. 
From (\cite{NPS}, Remark 1.3) we shall suppose that $i(\varphi)=1$, 
for $\varphi\in\mathcal{A}_0(\mathbb{K})$. So
from now on the support number of $\varphi_\varepsilon$ is equal to
$\varepsilon$ and instead of $\dd(\varphi_\varepsilon)$ we shall write
$\varepsilon$ only. This provides us  an unique extraction of
$\varepsilon$ from $\varphi_\varepsilon$ which is not the case in the
original Colombeau theory.  

Note that with above considerations we have that
$\dot{\alpha}_r\in\inv(\overline{\mathbb{K}}),~\forall~r\in\mathbb{R}$,
and for all $\varphi\in\mathcal{A}_0(\mathbb{K}),~i(\varphi)=1$, there
exists $\eta=\eta(\varphi)\in]0,1[$ such that
$$\lim\limits_{r\to\infty}\dot{\alpha}_r(\varphi_\varepsilon)=\lim\limits_{r\to\infty}(\varepsilon
i(\varphi))^r=0,~\forall~0<\varepsilon<\eta.$$ In this case, we say
that $\lim\limits_{r\to\infty}\dot{\alpha}_r=0$. We shall use this,
for example, in the proof of Lemma \ref{deriv-1} in Section \ref{sec-2}.

\subsection{The sharp topology on Colombeau's full generalized
numbers: a review}

In this subsection, we review some results and definitions about $\overline{\mathbb{K}}$,  
and we start with the following two definitions of \cite{JRJ} that are very
important for the definition of the topology on $\overline{\mathbb{K}}$.

 \begin{defn}
 An element $v\in\overline{\mathbb{K}}$ is associated to zero,
 $v\approx 0$, if for some (hence for each) representative
 $(v(\varphi))_{\varphi}$ of $v$ we have \[\exists
 ~p\in\mathbb{N}~\mbox{such that}~\lim_{\varepsilon\downarrow 0}
 v(\varphi_\varepsilon)=0,
 ~\forall~\varphi\in\mathcal{A}_p(\mathbb{K}).\] Two elements
 $v_1,v_2\in\overline{\mathbb{K}}$ are associated, $v_1\approx v_2$, if
 $(v_1-v_2)\approx 0$. If there exists $a\in\mathbb{K}$ with $v\approx
 a$, then $v$ is said to be associated with $a$ and the latter is
 called the shadow of $v$. 
 \end{defn}

 \begin{defn}
 For a given $x\in\overline{\mathbb{K}}$ we set
 $A(x)\coloneqq\{r\in\mathbb{R}:(\dot{\alpha}_{-r}x)\approx 0\}$ and
 define the valuation of $x$ as $V(x)=\sup(A(x))$.
 \end{defn}

 For the relation of association ``$\approx$'' on $\overline{\mathbb{K}}$
 see (\cite{JRJ}, Definition 1.3). It is easily seen that if
 $x\in\overline{\mathbb{K}}$, then $r\in A(x)\Leftrightarrow\exists~
 p\in\mathbb{N}$ such that 
 $\lim\limits_{\varepsilon\downarrow
   0}\varepsilon^{-r}x(\varphi_\varepsilon)=0,~\forall~\varphi\in\mathcal{A}_p(\mathbb{K})$
 or equivalently $\vert x\vert\le
 \dot{\alpha}_r,~\forall~\varphi_\varepsilon\le 1$ with $\varepsilon$
 sufficiently small. From this, it easily follows that
 $D:\overline{\mathbb{K}}\times\overline{\mathbb{K}}\to\mathbb{R}_+$
 defined by \[D(x,y)\coloneqq e^{-V(x-y)}\] is an ultra-metric on
 $\overline{\mathbb{K}}$ which is invariant under translations. The
 topology resulting from $D$ is so-called the sharp topology on
 $\overline{\mathbb{K}}$ and it is denoted by $\tau_s$. Denote 
the norm of an element $x\in \overline{\mathbb{K}}$  by $\Vert
x\Vert\coloneqq D(x,0)$. Thus, we have the distance between 
two elements $x,y\in\overline{\mathbb{K}}$ which is given by  
$D(x,y)\coloneqq\Vert x-y\Vert$.

 Now, we have the following result from \cite{JRJ}. 

 \begin{cor}\label{norm}
 For given
 $x,y\in\overline{\mathbb{K}},~r\in\mathbb{R},~s\in\mathbb{R}_+^*$ and
 $a,b\in\mathbb{K}$, we have:
 \begin{enumerate}
 \item[$i)$] $\Vert x+y\Vert\le\max(\{\Vert x\Vert,\Vert y\Vert\})$ and
   $\Vert xy\Vert\le \Vert x\Vert\Vert y\Vert$;
 \item[$ii)$] $\Vert x\Vert\ge 0$ and $\Vert x\Vert=0\Leftrightarrow
   x=0$;
 \item[$iii)$] $\Vert ax\Vert=\Vert x\Vert$, if $a\ne 0$;
 \item[$iv)$] $\Vert\dot{\alpha}_rx\Vert=e^{-r}\Vert x\Vert$ and
   $\Vert\dot{\beta}_sx\Vert=s\Vert x\Vert$, where $\dot{\beta}_s=\dot{\alpha}_{-\log(s)}$;
 \item[$v)$] $\Vert a\Vert=1$, if $a\ne 0$;
 \item[$vi)$] $\Vert a-b\Vert=1-\delta_{ab}$ (Kronecker's $\delta$).
 \end{enumerate}
 \end{cor}  

 Now, we remind  the following result of \cite{JRJ} which will be
 important to prove the  Proposition \ref{homom-1} in the  
 Section \ref{sec-3}.

 \begin{lem}\label{lema2.1}
 \begin{enumerate}
 \item[$i)$] $x\in B_1(0)\Leftrightarrow V(x)>0$;
 \item[$ii)$] If $x\in B_1(0)$, then $x\approx 0$ and $D(1,x)=1$. Hence,
   $1\notin\bar{B}_1(0),~B_1(0)\cap B_1(1)=\emptyset,~B_1'(0)\supset \bar{B}_1(0)$
   and $B_1'(0)\ne \bar{B}_1(0)$.
 \end{enumerate}
 \end{lem}

 The basic notation and some properties of the algebraic and the topological
 structure of $\overline{\mathbb{K}}$ can be found in \cite{JRJ}
 Let $\overline{\mathbb{K}}^{n}\coloneqq\{(x_1,x_2,\dots,x_n):
 x_i\in\overline{\mathbb{K}},~\forall~i=1,2,\dots
 n\}$ with the product topology, see \cite{JRJ} for more details. 
.

 If $x=(x_{1},x_{2},\ldots,x_{n})\in\overline{\mathbb{K}}^{n}$ we
 define $\Vert x\Vert_{n}\coloneqq\max\{\Vert x_{i}\Vert:1\leq i\leq n\}$,
 where $\Vert x_{i}\Vert$ is defined as before, and frequently  
 the subscript $n$ will be omitted from the  notation. 

 If $r\in\mathbb{R}_{+}^{*}$ and $x_{0}\in\overline{\mathbb{K}}^{n}$, then 
 $$B_{r}(x_{0})=\{x\in\overline{\mathbb{K}}^{n}:\Vert
 x-x_{0}\Vert<r\},\,B_{r}'(x_{0})=
 \{x\in\overline{\mathbb{K}}^{n}:\Vert x-x_{0}\Vert\leq r\}$$
 and $$S_{r}(x_{0})=\{x\in\overline{\mathbb{K}}^{n}:\Vert x-x_{0}\Vert=r\}$$
 are  the open ball of center in $x_{0}$ and ratio $r$,
 the closed ball of center in $x_{0}$ and ratio $r$ and the sphere
 of center in $x_{0}$ and ratio $r$, respectively.

\begin{rem}\label{lema2.1-1} It is convenient to point out that we 
easily  extend for $\overline{\mathbb{K}}^n$  the definitions of 
$B_r(x_0)$, $B'_r(x_0)$  and $S_r(x_0)$.
\end{rem}

\section{Differential Calculus over the ring of Colombeau's full 
generalized numbers}\label{sec-2}

We begin with the following lemma that will be fundamental 
to introduce the concept of the differentiable functions in the 
framework of Colombeau's full generalized numbers.

 \begin{lem}\label{deriv-1}
 Let $U\subset \overline{\mathbb{K}}$ be an open subset,
 $f:U\to\overline{\mathbb{K}}$ 
 a function and $x_0\in U$. Then there exists at most one $z_0\in\overline{\mathbb{K}}$, such that
 \begin{equation}\label{eq:deriv-1}
 \lim_{x\to x_0}\frac{f(x)-f(x_0)-z_0(x-x_0)}{\dot{\beta}_{\Vert
     x-x_0\Vert}}=0,
 \end{equation}
 where $\dot{\beta}_{\Vert x-x_0\Vert}=\dot{\alpha}_{-\log(\Vert
   x-x_0\Vert)}$ (as in Corolary \ref{norm}, item $iv)$, with $s=\Vert x-x_0\Vert$).
 \end{lem}

 \begin{proof}
 Let  $z_0,z_1$ be elements of $\overline{\mathbb{K}}$  such that the limit in (\ref{eq:deriv-1})
 is zero for both. Then  $$\lim_{x\to
   x_0}\frac{(z_1-z_0)(x-x_0)}{\dot{\beta}_{\Vert
     x-x_0\Vert}}=0.$$ 
 In particular, let $x_n\coloneqq x_0+\dot{\alpha}_n$. Then
 $\dot{\alpha}_n=x_n-x_0$ and $\dot{\beta}_{\Vert x_n-x_0\Vert}=
 \dot{\alpha}_{-\log(\Vert x_n-x_0\Vert)}=\dot{\alpha}_n$. 
 Thus,  we have that $$0=\lim_{x_n\to
   x_0}\frac{(z_1-z_0)(x_n-x_0)}{\dot{\beta}_{\Vert
     x_n-x_0\Vert}}=
 \lim_{n\to\infty}\frac{(z_1-z_0)\dot{\alpha}_n}{\dot{\alpha}_n}=
 \lim_{n\to\infty}(z_1-z_0)=z_1-z_0.$$ Hence, $z_1=z_0$. 
 \end{proof}

 The Lemma \ref{deriv-1} tell us that the following definition, which
 is the exact generalization of the Frechet derivative, is meaningful.

 \begin{defn}\label{deriv-2}
 Given an open set
 $U\subset\overline{\mathbb{K}},~f:U\to\overline{\mathbb{K}}$ and 
 $x_0\in U$ we shall say that $f$ is differentiable in $x_0$ if there 
 exists $z_0\in\overline{\mathbb{K}}$, such that the limit in
 (\ref{eq:deriv-1}) is valid. In this case $f$ is said to be
 differentiable in $x_0$, and we write $D(f)(x_0)=z_0$ and shall call
 $z_0$ the derivative of $f$ in $x_0$. We shall say that $f$ is
 differentiable if it is differentiable in each point of its domain.
 \end{defn}

 \begin{rem}\label{deriv-3}
 \begin{enumerate}
 \item[$a)$] The differentiability of $f$ in $x_0$ is equivalent 
 to the statement  that 
 \begin{equation}\label{eq:deriv-2}
 \lim_{x\to x_0}\frac{\Vert T(x)\Vert}{\Vert x-x_0\Vert}=0,
 \end{equation}
 where 
 \begin{equation}\label{eq:deriv-3}
 T(x)\coloneqq f(x)-f(x_0)-D(f)(x_0)(x-x_0),
 \end{equation}
 because by Corollary \ref{norm} item $iv)$, we have that
 \[\Vert\dot{\beta}_{\Vert
   x-x_0\Vert}\Vert=\Vert\dot{\alpha}_{-\log(\Vert
   x-x_0\Vert)}\Vert=e^{\log(\Vert x-x_0\Vert)}=\Vert x-x_0\Vert.\] 
 The choice of the limit that appears in (\ref{eq:deriv-1}) (instead of
 the limit in (\ref{eq:deriv-2})) follows from the necessity to avoid
 additional difficulties in the proof of some properties. Moreover, 
 by  the fact that $T(x)\in\overline{\mathbb{K}}$ we have that it  is natural to
 work, in the definition of the derivative, with a quotient of $T(x)$ 
 by an element of $\overline{\mathbb{K}}$ which is an "infinitesimal 
 together with $\Vert x-x_0\Vert$". Since  $\overline{\mathbb{K}}$
 is not a field, this infinitesimal must be an invertible element of 
 $\overline{\mathbb{K}}$ and we get that  the choice of
 $\dot{\beta}_{\Vert x-x_0\Vert}$ seems very natural since, 
 $\Vert\dot{\beta}_r\Vert=r,~\forall~r\in\mathbb{R}$.
 
 Note that the limit in (\ref{eq:deriv-2}) is not equivalent to 
 \begin{equation}\label{eq:deriv-4}
 \lim_{x\to x_0}\frac{T(x)}{\Vert x-x_0\Vert}=0,
 \end{equation}
 because of Corollary \ref{norm} item $iii)$, we have that 
 \begin{eqnarray*}
 \left\Vert\frac{T(x)}{\Vert x-x_0\Vert}\right\Vert&=&\left\Vert\frac{1}{\Vert x-x_0\Vert}T(x)\right\Vert\\
 &=&\Vert T(x)\Vert.
 \end{eqnarray*}
 Thus, if the differentiability of $f$ in $x_0$ is defined by the 
 limit in (\ref{eq:deriv-4}) with $T(x)$ as in (\ref{eq:deriv-3}), 
 then the continuity at  $x_0$ would imply its
 differentiability at $x_0$. Hence,  its derivative would be non-unique 
 because of any element of $\overline{\mathbb{K}}$ would be derivative
 of $f$ at $x_0$. So, this is not a good way to define
 differentiability in this context.
 \item[$b)$] If $f$ is differentiable in $x_0$, then we have that  
 \begin{equation}\label{eq:deriv-5}
 f(x)-f(x_0)=D(f)(x_0)(x-x_0)+E(x)
 \end{equation}
 with $\lim\limits_{x\to x_0}\frac{E(x)}{\dot{\beta}_{\Vert
     x-x_0\Vert}}=0$. Moreover, 
$D(f)(x_0)=\lim\limits_{n\to\infty}\frac{f(x_0+\dot{\alpha}_n)-f(x_0)}{\dot{\alpha}_n}$.
 \end{enumerate}
 \end{rem}

 From now on we use Remark \ref{deriv-3} item $b)$ without further mention.

 \begin{lem}\label{deriv-4}
 Let $U\subset\overline{\mathbb{K}}$ be an open subset. If
 $f:U\to\overline{\mathbb{K}}$ 
 is differentiable at $x_0$, then $f$ is continuous at $x_0$.
 \end{lem}

 \begin{proof}
 Since \[\lim\limits_{x\to
   x_0}\frac{E(x)}{\dot{\beta}_{\Vert x-x_0\Vert}}=0\] 
 then for any $\varepsilon>0$ there exists $\delta>0$, 
 such that $\left\Vert \frac{E(x)}{\dot{\beta}_{\Vert
       x-x_0\Vert}}\right\Vert<\varepsilon$ 
 always that $\Vert x-x_0\Vert<\delta$. Note that 
 \begin{eqnarray*}
 \left\Vert \frac{E(x)}{\dot{\beta}_{\Vert
   x-x_0\Vert}}\right\Vert&=&\left\Vert\frac{E(x)}{\Vert x-x_0\Vert}\right\Vert\\
 &=&\left\Vert\frac{1}{\Vert x-x_0\Vert}E(x)\right\Vert\\
 &=&\Vert E(x)\Vert
 \end{eqnarray*}
 where the  last equality is due to Corollary \ref{norm} item $iii)$.
 Thus, we have that  $\Vert E(x)\Vert<\varepsilon$ always that $\Vert
 x-x_0\Vert<\delta$ which implies that $\lim\limits_{x\to x_0}E(x)=0$ 
 and the result follows.
 \end{proof}

 We now give an example of a non-constant function whose derivative
 vanishes everywhere. Hence a function that is not determined by its derivative.
 This example also shows that the ``Mean Value Theorem'' is 
false in general in our context, thus as in \cite{OJR}.

 \begin{exam}\label{deriv-5}
 Let $$f(x)=\left\{\begin{array}{ll}
 \dot{\beta}_{\Vert x\Vert^2}, &\mbox{if $x\in\overline{\mathbb{K}}^*$}\\
 0, &\mbox{if $x=0$}.
 \end{array}\right.$$
 If $x_0\ne 0$, then $f$ is constant in the neighborhood $S_{\Vert
   x_0\Vert}$ of $x_0$ and we have that $D(f)(x_0)=0$ because of 
 \begin{eqnarray*}
 D(f)(x_0)&=&\lim_{n\to\infty}\frac{f(0+\dot{\alpha}_n)-f(0)}{\dot{\alpha}_n}\\
 &=&\lim_{n\to\infty}\frac{f(\dot{\alpha}_n)}{\dot{\alpha}_n}\\
 &=&\lim_{n\to\infty}\frac{\dot{\alpha}_{-2\log\Vert\dot{\alpha}_n\Vert}}{\dot{\alpha}_n}\\
 &=&\lim_{n\to\infty}\frac{\dot{\alpha}_{2n}}{\dot{\alpha}_n}\\
 &=&\lim_{n\to\infty}\dot{\alpha}_n\\
 &=&0.
 \end{eqnarray*}
 \end{exam}

 Now, it is convenient to point out that a function will be 
 called almost constant if it has vanishing derivative.

 Using the Remark \ref{deriv-3} and the standard proofs of ordinary
 differential calculus we obtain the following result.

 \begin{prop}\label{deriv-6}
 Let $U\subset \overline{\mathbb{K}}$ be an open subset. If
 $f,g:U\to\overline{\mathbb{K}}$ be differentiable, then
 \begin{enumerate}
 \item[$a)$] $fg$ is differentiable and $D(fg)=D(f)g+fD(g)$;
 \item[$b)$] If $f(U)$ is contained in the domain of $g$, then
   $D(f\circ g)=(D(g)\circ f)D(f)$;
 \item[$c)$] $D(f\pm g)=D(f)\pm D(g)$ and $D(cf)=cD(f)$, if $c$ is constant;
 \item[$d)$] if $g(x)\in\inv(\overline{\mathbb{K}}),~\forall~x\in U$ we
   have that $D\left(\dfrac{f}{g}\right)=\dfrac{D(f)g-fD(g)}{g^2}$.
 \end{enumerate}
 \end{prop}

 The Proposition \ref{deriv-6} tell us that our notion of derivations
 satisfies the usual properties of the derivation of ordinary
 differential calculus. 

 Next, we introduce the differentiability in functions with more than one variable.

 \begin{defn}\label{deriv-7}
 Let $U\subset \overline{\mathbb{R}}^n$ be an open subset,
 $f:U\to\overline{\mathbb{K}}$ a function, 
 $x=(x_1,x_2,\dots,x_n),x_0=(x_{01},x_{02},\dots,x_{0n})\in U$.  
Suppose that there exists an element
 $a_i\in\overline{\mathbb{K}}$, such that 
 \begin{equation}\label{eq:deriv-6}
 \lim_{h\to 0}\frac{f(x_1,x_2,\dots,x_i+h,\dots,x_n)-f(x_{01},x_{02},
 \dots,x_{0i},\dots,x_{0n})-a_ih}{\dot{\beta}_{\Vert h\Vert}}=0.
 \end{equation}
 Then we shall define $\dfrac{\partial f}{\partial x_i}(x_0)\coloneqq
 a_i$ and call it the partial derivative of $f$ with respect to $x_i$
 in $x_0$. We shall say that $f$ is differentiable in $x_0$ if there
 exists a vector $a=(a_1,a_2,\dots, a_n)\in\overline{\mathbb{K}}^n$, such
 that 
\begin{equation}\label{eq:deriv-7}
\lim_{x\to x_0}\frac{f(x)-f(x_0)-(a_1,a_2,\dots,a_n)\cdot
   (h_1,h_2,\dots,h_n)}{\dot{\beta}_{\Vert x-x_0\Vert}}=0,
\end{equation}
 where $h_i=x_i-x_{0i}, ~i=1,2\dots,n$ are the components of the
 difference vector $h$.
 \end{defn}

 It is now standard to verify that all the known results of ordinary
 differential calculus hold also in our case. For example, if $f$ is
 differentiable in $x_0$, then it is continuous in $x_0$ and the
 $a_i$'s in the definition of $f$ being differentiable are exactly the
 partial derivatives in $x_0$. If $\mathbb{K}=\mathbb{R}$, the gradient
 of $f$ at $x_0$ is defined by the vector 
 \begin{equation}\label{eq:deriv-8}
 \nabla(f)(x_0)\coloneqq\left(\frac{\partial f}{\partial
     x_1}(x_0),\frac{\partial f}{\partial
     x_2}(x_0),\dots,\frac{\partial f}{\partial x_n}(x_0)\right), 
 \end{equation}
 where $\frac{\partial f}{\partial x_i}=a_i$, for $i=1,2,\dots,n$.

 If $U$ is an open subset of $\overline{\mathbb{R}}^n$ and
 $k\in\mathbb{N}$, we can define the set
 $$\mathcal{C}^k(U;\overline{\mathbb{K}})\coloneqq\{f:U\to\overline{\mathbb{K}}|\partial^\alpha
 f\in\mathcal{C}(U;\overline{\mathbb{K}}),~\forall~\alpha\in\mathbb{N}^n~\mbox{such
 that} ~0\le |\alpha|\le k\}$$ and 
 $\mathcal{C}^\infty(U;\overline{\mathbb{K}})\coloneqq
 \underset{k\in\mathbb{N}}{\cap}\mathcal{C}^k(U;\overline{\mathbb{K}})$.

 \begin{rem}\label{deriv-8}
 \begin{enumerate}
 \item[$1)$] Let $U\subset\overline{\mathbb{R}}^n$ be an open subset and
 $f:U\to\overline{\mathbb{R}}^m$ a function. We may write
 $f=(f_1,f_2,\dots,f_m)$, where each
 $f_i:U\to\overline{\mathbb{R}}$, $i=1,2,\dots,m$ is 
  the coordinated function of $f$. It is convenient to point out that  the differentiability of $f$ 
  at $x_0\in U$ is equivalent to   $f_i$ be  differentiable at
  $x_{0i}$, for all $i=1,...,m$, where $x_0=(x_{01},...,x_{0n})$.
 \item[$2)$] It is easy to see that $f$ is differentiable at $x_0$ if and only if  there exists
 a $\overline{\mathbb{R}}$-linear map
 $T:\overline{\mathbb{R}}^n\to\overline{\mathbb{R}}^m$, such that 
 \begin{equation}\label{eq:deriv-8}
 \lim_{x\to x_0}\frac{f(x)-f(x_0)-T(x-x_0)}{\dot{\beta}_{\Vert x-x_0\Vert}}=0.
 \end{equation}
 The map $T$ will be denoted by $D(f)(x_0)$ and 
 \begin{equation}\label{eq:deriv-9}
 J=[T]_{m\times n}=\left[\begin{array}{cccc}
 \frac{\partial f_1}{\partial x_1} & \frac{\partial f_1}{\partial x_2}
                    & \dots & \frac{\partial f_1}{\partial x_n}\\
 \frac{\partial f_2}{\partial x_1} & \frac{\partial f_2}{\partial x_2}
                    & \dots & \frac{\partial f_2}{\partial x_n}\\
 \vdots & \vdots & \ddots & \vdots\\
 \frac{\partial f_m}{\partial x_1} & \frac{\partial f_m}{\partial x_2}
                    & \dots & \frac{\partial f_m}{\partial x_n}
 \end{array}\right] 
 \end{equation}
 is the Jacobian matrix.
 \end{enumerate}
 \end{rem}

\section{Pointvalues and generalized numbers}\label{sec-3}
 Within classical distribution theory a definition of pointvalues for
 distributions was introduced in \cite{SL} and see also
 \cite{OMM}. However, this concept cannot be applied to arbitrary distributions
 at arbitrary points. Moreover, there is no way of characterizing distributions
 by their pointvalues in any way similar to classical functions. On
 the other hand, for elements of Colombeau algebras there is a very
 natural way of obtaining pointvalues by inserting
 points into representatives. The objects gained from such  operation
 are sequences of numbers and then are not values in the field $\mathbb{K}$, but 
 they are representatives of generalized numbers. Our first
 aim will be to gain an exact description of these objects. 
 In this section we take some ideas of \cite{MK} to generalize some
 results about Colombeau's simplified generalized function for
 Colombeau's full generalized function.

It is convenient to point out at this point that $\mathbb{K}$ is embedded into 
$\overline{\mathbb{K}}$ via $c\mapsto\cl[c(\varphi)]$
with $c(\varphi)=c,\,\forall\,\varphi\in\mathcal{A}_{0}(\mathbb{K})$ and 
$\overline{\mathbb{K}}$ is the natural home of pointvalues of elements
of $\mathcal{G}(\Omega)$. Note that by an analogous proof to that
given in \cite{JFC4} demonstrates that $\overline{\mathbb{K}}_{s}$ 
can be a subring of $\overline{\mathbb{K}}$. Moreover,
$\overline{\mathbb{K}}$ is the ring of constants in $\mathcal{G}(\Omega)$.

We begin this section with the following definition. 

\begin{defn}
Let $U\in\mathcal{G}(\Omega)$ and $x\in\Omega$. The pointvalue of
$U$ at $x$ is the element $cl[(u_{\varphi}(x))_{\varphi}]$ of
$\overline{\mathbb{K}}$, where $u_{\varphi}$ is one of the representatives of $U$.
\end{defn}
 
\begin{prop}
 Let $\Omega$ be a connected open subset of $\mathbb{R}^{n}$ and
 $U\in\mathcal{G}(\Omega)$. Then $\nabla U\equiv0$ if and only if
 $U\in\overline{\mathbb{K}}$.
\end{prop}

\begin{proof}
Evidently, $U\in\overline{\mathbb{K}}$ implies $\nabla U\equiv0$.
Conversely, let $(\nabla u_{\varphi})_{\varphi}\in\mathcal{N}(\Omega)^{n}$.
Then, we assume that $\Omega$ is star-shaped and we have that  there 
exists $p\in\mathbb{N},\gamma\in\Gamma$ and $c=c_{\varphi}>0$, such
that 
 \[
 \vert\nabla u_{\varphi_{\varepsilon}}(x)\vert\leq
 c\varepsilon^{\gamma(q)-p},\,
\forall\,\varphi\in\mathcal{A}_{q}(\mathbb{K}),\,q\geq p\]
and $\varepsilon$ sufficiently small. Hence, 

 \begin{eqnarray*}
 |u_{\varphi_{\varepsilon}}(x)-u_{\varphi_{\varepsilon}}(x_{0})| & =
 & |(x-x_{0})\intop_{0}^{1}\nabla u_{\varphi_{\varepsilon}}(x-\sigma(x-x_{0}))d\sigma|\\
  & \leq & |x-x_{0}|\intop_{0}^{1}|\nabla u_{\varphi_{\varepsilon}}(x-\sigma(x-x_{0}))|d\sigma\\
  & \leq & |x-x_{0}|c\varepsilon^{\gamma(q)-p}, \forall\,\varphi\in\mathcal{A}_{q}(\mathbb{K}),\,q\geq p
 \end{eqnarray*}
 for arbitrary $p$ and suitable $q$. Thus $((u_{\varphi})_{\varphi}-(u_{\varphi}(x_{0}))_{\varphi})\in\mathcal{N}(\mathbb{K})$.
 Now, if $\Omega$ is connected, then any point $x\in\Omega$ can be connected
 with some fixed $x_{0}\in\Omega$ by a polygon and for an analogous argument
 to the one above we get the result. 
 \end{proof}

Let $\Omega$ be an open subset of $\mathbb{K}^{n}$ and $I_{\eta}$ as
in Notation \ref{nota} in the item $a)$. Define 
 \[
 \Omega_{M}\coloneqq\{(x_{\varphi})\in\Omega^{\mathcal{A}_{0}(\mathbb{K})}:\exists\,p\in\mathbb{N}\,
 \mbox{s.t.}\,\forall\,\varphi\in\mathcal{A}_{p}(\mathbb{K}),\,\exists\,c=c_{\varphi}>0,\,
 \mbox{s.t.}\,|x_{\varphi_{\varepsilon}}|\leq c\varepsilon^{-p},\,\varepsilon\in I_{\eta}\}
 \]
 and we introduce the  equivalence relation defined by 
 \[
 (x_{\varphi})_{\varphi}\sim(y_{\varphi})_{\varphi}\Leftrightarrow\exists\,p\in\mathbb{N}\,
 \gamma\in\Gamma,\,\exists\,c=c_{\varphi}>0,\,
 \mbox{s.t.}\,|x_{\varphi_{\varepsilon}}-y_{\varphi_{\varepsilon}}|\leq c\varepsilon^{\gamma(q)-p},\,\varepsilon\in I_{\eta}
 \]
 $\forall\,\varphi\in\mathcal{A}_{q}(\mathbb{K}),\,q\geq p$.  Set
 $\tilde{\Omega}\coloneqq\Omega_{M}/\sim$. 

The set of compactly supported points is 
 \[
 \tilde{\Omega}_{c}=\{\tilde{x}\in\tilde{\Omega}:\exists\,\mbox{repres.}\,(x_{\varphi})_{\varphi},
 \,\exists\,K\subset\subset\Omega,\exists\,p\in\mathbb{N}\,\mbox{s.t.}
 \,x_{\varphi_{\varepsilon}}\in
 K,\,\forall\,\varphi\in\mathcal{A}_{p}(\mathbb{K}),
 \,\forall\,\varepsilon\in I_{\eta}\}.
 \]
 It is clear that if the $\tilde{\Omega}_{c}-$property holds for one
 representative of $\tilde{x}\in\tilde{\Omega}$, then it holds for
 every representative. Also, for $\Omega=\mathbb{K}$, we have $\tilde{\mathbb{K}}=\overline{\mathbb{K}}$.
 Thus, we have that the canonical identification $\tilde{\mathbb{K}^{n}}=\tilde{\mathbb{K}}^{n}=\overline{\mathbb{K}}^{n}$.
 For $\tilde{\mathbb{K}}_{c}$ we write $\overline{\mathbb{K}}_{c}$.

From considerations above, we have the following result.

\begin{prop}\label{domain}
 Let $U\in\mathcal{G}(\Omega)$ and $\tilde{x}\in\tilde{\Omega}_{c}$.
 Then the generalized pointvalue of $U$ at 
 $\tilde{x}=\cl[(x_{\varphi})_{\varphi}]$ is $U(\tilde{x})\coloneqq\cl[(u_{\varphi}(x_{\varphi}))_{\varphi}]$. Moreover,  it 
 is a well-defined element of $\overline{\mathbb{K}}$.
 \end{prop}

 \begin{proof}
 If $\tilde{x}\in\tilde{\Omega}_{c}$, then there exists $K\subset\subset\Omega,\,p\in\mathbb{N}$,
 such that $x_{\varphi_{\varepsilon}}\in K,\,\forall\,\varphi\in\mathcal{A}_{p}(\mathbb{K}),\,\forall\,\varepsilon\in I_{\eta}$.
 Since $U\in\mathcal{G}(\Omega)$ then we have that
 \[
 |u_{\varphi_{\varepsilon}}(x_{\varphi_{\varepsilon}})|\leq\sup_{x\in
   K}|u_{\varphi_{\varepsilon}}(x_{\varphi_{\varepsilon}})|
 \leq c\varepsilon^{-p},\,\forall\,\varphi\in\mathcal{A}_{p}(\mathbb{K}),\,\forall\,\varepsilon\in I_{\eta}
 \]
 and for some $c=c_{\varphi}>0$. Next we show that $\tilde{x}\sim\tilde{y}\Rightarrow U(\tilde{x})\sim U(\tilde{y})$,
 i.e., $x_{\varphi}\sim y_{\varphi}\Rightarrow u_{\varphi}(x_{\varphi})\sim u_{\varphi}(y_{\varphi})$ and we need  
  to prove  that there exists $p'\in\mathbb{N},\,\gamma\in\Gamma$
 and $c'=c'_{\varphi}>0$, such that 
 \begin{equation}\label{eq:equiva-1}
 |u_{\varphi_{\varepsilon}}(x_{\varphi_{\varepsilon}})-u_{\varphi_{\varepsilon}}(y_{\varphi_{\varepsilon}})|
 \leq c'\varepsilon^{\gamma(q)-p'},\,\forall\,\varphi\in\mathcal{A}_{q}(\mathbb{K}),\,q
 \geq p',\,\varepsilon\in I_{\eta}.
 \end{equation}
Note that 
\begin{eqnarray}\label{eq:equiva-2}
 |u_{\varphi_{\varepsilon}}(x_{\varphi_{\varepsilon}})-u_{\varphi_{\varepsilon}}(y_{\varphi_{\varepsilon}})|
 &=&|(x_{\varphi_{\varepsilon}}-y_{\varphi_{\varepsilon}}\intop_{0}^{1}\nabla
     u_{\varphi_{\varepsilon}}(x_{\varphi_{\varepsilon}}-
 \sigma(x_{\varphi_{\varepsilon}}-y_{\varphi_{\varepsilon}}))d\sigma|\nonumber\\
 &\leq&
        |x_{\varphi_{\varepsilon}}-y_{\varphi_{\varepsilon}}|\intop_{0}^{1}|
 \nabla u_{\varphi_{\varepsilon}}(x_{\varphi_{\varepsilon}}-\sigma(x_{\varphi_{\varepsilon}}-y_{\varphi_{\varepsilon}}))|d\sigma.
 \end{eqnarray}
Since $x_{\varphi}\sim y_{\varphi}$ then there exists $p''\in\mathbb{N},\,\gamma'\in\Gamma$
 and $c''=c''_{\varphi}>0$, such that 
 \begin{equation}\label{eq:equiva-3}
 |x_{\varphi_{\varepsilon}}-y_{\varphi_{\varepsilon}}|\leq
 c''\varepsilon^{\gamma'(q')-p''},\,
 \forall\,\varphi\in\mathcal{A}_{q'}(\mathbb{K}),\,q'\geq p'',\,\varepsilon\in I_{\eta}.
 \end{equation}
By the fact that  $x_{\varphi_{\varepsilon}}-\sigma(x_{\varphi_{\varepsilon}}-y_{\varphi_{\varepsilon}})$
 remains within some compact subset of $\Omega$ for $\varepsilon\in I_{\eta}$
 and $U\in\mathcal{G}(\Omega)$, we have that 
 \begin{equation}\label{eq:equiva-4}
 |\nabla u_{\varphi_{\varepsilon}}(x_{\varphi_{\varepsilon}}-\sigma(x_{\varphi_{\varepsilon}}-y_{\varphi_{\varepsilon}}))|
 \leq\sup_{x\in K}|\nabla u_{\varphi_{\varepsilon}}(x)|\leq
 c'''\varepsilon^{-p'''},\,\forall\,
 \varphi\in\mathcal{A}_{p'''}(\mathbb{K}),\,\forall\,\varepsilon\in I_{\eta}
 \end{equation}
  Replacing (\ref{eq:equiva-3}) and (\ref{eq:equiva-4}) in (\ref{eq:equiva-2}),
 we have that 
 \begin{eqnarray*}
 |u_{\varphi_{\varepsilon}}(x_{\varphi_{\varepsilon}})-u_{\varphi_{\varepsilon}}(y_{\varphi_{\varepsilon}})|
   &\leq& (c''c''')\varepsilon^{\gamma'(q')-(p''+p''')}=c'\varepsilon^{\gamma(q)-p},\,\forall\,
 \varphi\in\mathcal{A}_{q'}(\mathbb{K}),\,q'\geq(p''+p'''),\,\varepsilon\in I_{\eta},
 \end{eqnarray*}
 where $p'=p''+p''',\,\gamma(q)=\gamma'(q')$ and $c'=c''c'''$.  Thus the 
 inequality in (\ref{eq:equiva-1}) holds which implies that 
 $\tilde{x}\sim\tilde{y}\Rightarrow U(\tilde{x})\sim U(\tilde{y})$.
 Next, if $(w_{\varphi}(x_{\varphi}))_{\varphi}\in\mathcal{N}(\Omega)$,
 then $w_{\varphi}(x_{\varphi})\sim0$,  because of $\exists\,p\in\mathbb{N},\,\gamma\in\Gamma$
 and $c=c_{\varphi}>0$, such that 
 \[
 |w_{\varphi_{\epsilon}}(x_{\varphi_{\varepsilon}})|\leq
 c\varepsilon^{\gamma(q)-p},
 \,\varphi\in\mathcal{A}_{q}(\mathbb{K}),\,q\geq p,\,\varepsilon\in I_{\eta}
 \] and we have that $x_{\varphi_{\varepsilon}}$ stays within some compact subset of $\Omega$
 for $\varepsilon\in I_{\eta}$.
 \end{proof}

Next, we study when the elements in $\mathcal{G}(\Omega)$ are
identically zero, i.e, we can characterize full generalized
functions in $\tilde{\Omega}_c$ from their point values as  in classical
functions. 

\begin{thm}\label{caracter}
 If $\Omega$ is an open subset of $\mathbb{R}^n$, then \[U\equiv
 0~\mbox{in}~\mathcal{G}(\Omega)\Leftrightarrow U(\tilde{x})=0
 ~\mbox{in} ~\overline{\mathbb{K}},~\forall~\tilde{x}\in\tilde{\Omega}_c.\]
 \end{thm}

 \begin{proof}
 $(\Rightarrow)$ We suppose that $U\equiv 0$ in $\mathcal{G}(\Omega)$
 and  we  show that $U(\tilde{x})=0$ in $\overline{\mathbb{K}}$ for
 all $\tilde{x}\in\tilde{\Omega}_c$. In fact, let 
 $u_\varphi,u_\varphi(x_\varphi)$, where $x_\varphi$ is a
 representative of $\tilde{x}$,  $u_{\varphi}$ a  representative of
 $U\in\mathcal{G}(\Omega)$ and $u_{\varphi}(x_{\varphi})$ is a representative of $U(\tilde{x})$ in
 $\overline{\mathbb{K}}$.  Since
 $u_\varphi\in\mathcal{N}(\Omega)$ and $x_\varphi$ is a representative
 of $\tilde{x}\in\tilde{\Omega}_c$,  then there exists
 $K\subset\subset\Omega\ni x_{\varphi_\varepsilon},
 ~p\in\mathbb{N},~c=c_\varphi>0$, such
 that \[|u_{\varphi_\varepsilon}(x_{\varphi_\varepsilon})|\le\sup_{x\in
   K}|u_{\varphi_\varepsilon}(x)|\le
 c\varepsilon^{\gamma(q)-p},~\forall~\varphi\in\mathcal{A}_q(\mathbb{K}),~q\ge
 p,~\varepsilon\in I_\eta\] for some $\gamma\in\Gamma$. Hence,
 $u_\varphi(x_\varphi)\in\mathcal{N}(\mathbb{K}), ~\forall~x_\varphi$
 representative of $\tilde{x}\in\tilde{\Omega}_c$,
 i.e. $U(\tilde{x})=0$ in
 $\overline{\mathbb{K}},~\forall~\tilde{x}\in\tilde{\Omega}_c$.

 $(\Leftarrow)$ If $U\ne 0$ in $\mathcal{G}(\Omega)$, then there
 exists $K\subset\subset\Omega, ~\alpha\in\mathbb{N}^n$ such that
 $\forall~p_1\in\mathbb{N},~\forall~\gamma\in\Gamma,~\exists~\varphi\in\mathcal{A}_{q_1}(\mathbb{K}),~q_1\ge
 p_1$ such that $\forall~c_1=c_{1\varphi}\ge 0$, we have that
 \begin{equation}\label{eq:equiva-5}
 \sup_{x\in K}|\partial^\alpha
 u_{\varphi_\varepsilon}(x)|>c_1\varepsilon^{\gamma(q_1)-p_1},~\forall~
 \varepsilon\in I_\eta
 \end{equation} 
 We choose $\alpha$ with the above property in such a way that
 $|\alpha|$ is minimal. Then, (\ref{eq:equiva-5}) yields the existence
 of sequences $\varepsilon_k\to 0$ and $x_{\varphi_{\varepsilon_k}}\in
 K$ such that 
 \begin{equation}\label{eq:equiva-6}
 |\partial^\alpha
 u_{\varphi_{\varepsilon_k}}(x_{\varphi_{\varepsilon_k}})|>
 c_1\varepsilon_k^{\gamma(q_1)-p_1},~\forall~k\in\mathbb{N}.
 \end{equation}
 Let
 $\varepsilon>0$ and   we set
 $x_{\varphi_\varepsilon}=x_{\varphi_{\varepsilon_k}}$ for
 $\varepsilon_{k+1}<\varepsilon<\varepsilon_k, k\in\mathbb{N}$. Then,
 $(x_\varphi)_\varphi\in\Omega_M$ and it has values in $K$. Consequently, 
 $\tilde{x}=\cl[(x_\varphi)_\varphi]$ belongs to
 $\tilde{\Omega}_c$. Also, from  the equation
 (\ref{eq:equiva-6}), we have that $\partial^\alpha U(\tilde{x})\ne 0$
 in $\overline{\mathbb{K}}$. Now, we have the following two cases:
 \begin{enumerate}
 \item[$i)$] $\alpha=0$;
 \item[$ii)$] $\alpha\ne 0$.
 \end{enumerate} 

 $i)$ If $\alpha=0$, then $U(\tilde{x})\ne 0$ in
 $\overline{\mathbb{K}}$ and the result follows.

 $ii)$ If $\alpha\ne 0$, we  show that this leads to a
 contradiction. Indeed, since
 $|\alpha|=|(\alpha_1,\alpha_2,\dots,\alpha_n)|$ was assumed to be
 minimal, then for any $\beta\in\mathbb{N}^n$ with $|\beta|=|\alpha|-1$ and
 $L\subset\subset\Omega$, we have that there exists
 $p_2\in \mathbb{N}, \gamma\in\Gamma$, such that for any
 $\varphi\in\mathcal{A}_{q_2}(\mathbb{K}),~q_2\ge p_2$ there exists
 $c_2=c_{2\varphi}>0$, such that 
 \begin{equation}\label{eq:equiva-7}
 \sup_{x\in L}|\partial^\beta u_{\varphi_\varepsilon}(x)|\le
 c_2\varepsilon^{\gamma(q_2)-p_2},~\forall ~\varepsilon\in I_\eta.
 \end{equation} 
 Now, we may assume that $\alpha_1\ne 0$ in
 $\alpha=(\alpha_1,\alpha_2,\dots,\alpha_n)$.  Let
 $\beta:=(\alpha_1-1,\alpha_2,\dots,\alpha_n),
 ~\beta'=(\alpha_1+1,\alpha_2,\dots,\alpha_n)$ and $x=(x_1,x')$ with
 $x'=(x_2,\dots,x_n)\in\mathbb{R}^{n-1}$.  Since
 $(u_\varphi)_\varphi\in\mathcal{E}_M(\Omega)$, then  we get  that there
 exists $p_3\in\mathbb{N}$, $c_3=c_{3\varphi}>0$  and  for any
 $\varphi\in\mathcal{A}_{p_3}(\mathbb{K})$ we have that  
 \begin{equation}\label{eq:equiva-8}
 \sup_{x\in L}|\partial^{\beta'}u_{\varphi_\varepsilon}(x)|\le
 c_3\varepsilon^{-p_3},~\forall ~\varepsilon\in I_\eta.
 \end{equation} 
 Choose $L\subset\subset\Omega$ such that $K\subset L^0$, where $L^0$
 denotes the interior of $L$. Then, for $k$ sufficiently large, we have
 that
 \begin{eqnarray}\label{eq:equiva-9}
 |\partial^\alpha
   u_{\varphi_{\varepsilon_k}}(y_{1\varphi_{\varepsilon_k}},x_{2\varphi_{\varepsilon_k}},\dots,x_{n\varphi_{\varepsilon_k}})|
 &=&\left|\partial^\alpha
     u_{\varphi_{\varepsilon_k}}(x_{\varphi_{\varepsilon_k}})+
 \intop_{x_{1\varphi_{\varepsilon_k}}}^{y_{1\varphi_{\varepsilon_k}}}
 \partial^{\beta'}u_{\varphi_{\varepsilon_k}}(\xi,x'_{\varphi_{\varepsilon_k}})d\xi\right|\nonumber\\
 &\ge&|\partial^\alpha
       u_{\varphi_{\varepsilon_k}}(x_{\varphi_{\varepsilon_k}})|-
 |y_{1\varphi_{\varepsilon_k}}-x_{1\varphi_{\varepsilon_k}}|c_3\varepsilon_k^{-p_3}\nonumber\\
 &\ge&c_1\varepsilon_k^{\gamma(q_1)-p_1}-|y_{1\varphi_{\varepsilon_k}}-x_{1\varphi_{\varepsilon_k}}|c_3\varepsilon_k^{-p_3},
 \end{eqnarray} 
   because of  the inequalities (\ref{eq:equiva-6})
 and (\ref{eq:equiva-8}). For this $k$ (sufficiently large), we have
 that $(y_{1\varphi},x_{2\varphi},\dots,x_{n\varphi})$ and
   $(x_{1\varphi},x_{2\varphi},\dots,x_{n\varphi})$ belongs $K\subset
   L^0\subset L$ and we get that 
   $(y_{1\varphi},x_{2\varphi},\dots,x_{n\varphi})\sim
   (x_{1\varphi},x_{2\varphi},\dots,x_{n\varphi})$. Consequently,  
   there exists $p_4\in\mathbb{N}$, $c_4=c_{4\varphi}>0$ with $q_3\geq p_4$, $\gamma\in\Gamma$, such that for any
   $\varphi\in\mathcal{A}_{q_3}(\mathbb{K})$  we have that 
 \begin{equation}\label{eq:equiva-10}
 |y_{1\varphi_{\varepsilon_k}}-x_{1\varphi_{\varepsilon_k}}|\le
 c_4\varepsilon_k^{\gamma(q_3)-p_4}, ~\forall~\varepsilon_k\in I_\eta.
 \end{equation} 
 Replacing the inequality (\ref{eq:equiva-10}) in
 (\ref{eq:equiva-9}), observing the signal, we obtain that 
 \begin{equation}
 |\partial^\alpha
   u_{\varphi_{\varepsilon_k}}(y_{1\varphi_{\varepsilon_k}},x_{2\varphi_{\varepsilon_k}},\dots,x_{n\varphi_{\varepsilon_k}})|\ge 
 c_1\varepsilon_k^{\gamma(q_1)-p_1}-c_4\varepsilon_k^{\gamma(q_3)-p_4}c_3\varepsilon_k^{-p_3}.
 \end{equation}
 Now, setting
 $\bar{x}_{\varphi_{\epsilon_k}}:=(x_{\varphi_{\varepsilon_k}}+c_4\varepsilon_k^{\gamma(q_3)-p_4},x'_{\varphi_{\varepsilon_k}})$,
 we have that
 \begin{eqnarray}\label{eq:equiva-11}
 |\partial^\beta u_{\varphi_{\varepsilon_k}}(\bar{x}_{\varphi_{\epsilon_k}})|&=&\left|\partial^\beta
   u_{\varphi_{\varepsilon_k}}(x_{\varphi_{\varepsilon_k}})+\intop_{x_{1\varphi_{\varepsilon_k}}}^{\bar{x}_{1\varphi_{\varepsilon_k}}}\partial^\alpha
   u_{\varphi_{\varepsilon_k}}(\xi,x'_{\varphi_{\varepsilon_k}}))d\xi\right|\nonumber\\
 &\ge&-c_2\varepsilon_k^{\gamma(q_2)-p_2}+|\bar{x}_{1\varphi_{\varepsilon_k}}-x_{1\varphi_{\varepsilon_k}}|c_1\varepsilon_k^{\gamma(q_1)-p_1},
 \end{eqnarray}
 because of  the inequalities (\ref{eq:equiva-6})
 and (\ref{eq:equiva-7}). Note that 
 \begin{equation}\label{eq:equiva-12}
 |\bar{x}_{1\varphi_{\varepsilon_k}}-x_{1\varphi_{\varepsilon_k}}|=c_4\varepsilon_k^{\gamma(q_3)-p_4},
 \end{equation}
 since  
 $\bar{x}_{\varphi_{\epsilon_{k}}}-x_{\varphi_{\varepsilon_k}}=(c_4\varepsilon_k^{\gamma(q_3)-p_4},0,\dots,0)$. 
 Replacing,  (\ref{eq:equiva-12}) in (\ref{eq:equiva-11}), we obtain
 \begin{equation}\label{eq:equiva-13}
 |\partial^\beta
 u_{\varphi_{\varepsilon_k}}(\bar{x}_{\varphi_{\epsilon_k}})|\ge
 -c_2\varepsilon_k^{\gamma(q_2)-p_2}+c_4\varepsilon_k^{\gamma(q_3)-p_4}c_1\varepsilon_k^{\gamma(q_1)-p_1}.
 \end{equation}
 Note that  for $\gamma(q_2)-p_2$ large enough  and $k>k_0$  we have that \[\sup_{x\in L}|\partial^\beta
 u_{\varphi_{\varepsilon_k}}(x)|\ge
 -c_2\varepsilon_k^{\gamma(q_2)-p_2}+c_4\varepsilon_k^{\gamma(q_3)-p_4}c_1\varepsilon_k^{\gamma(q_1)-p_1}\underset{\varepsilon_k\to
   0}{\to} 0,\] 
this is a contradiction  because of the inequality
(\ref{eq:equiva-13}). Therefore, $\alpha=0$ and the result follows.
 \end{proof}

It is convenient to point out that the Proposition \ref{domain} and 
Theorem \ref{caracter} say that the true domain of the full
generalized functions are the sets $\tilde{\Omega}_c$. 

We now extend the definition of association presented 
(\cite{JRJ},  Definition 1.3) to $\overline{\mathbb{K}}^n$: An element
 $x=(x_1,x_2,\dots,x_n)\in\overline{\mathbb{K}}^n$ is associated to
 $0=(0,0\dots,0)$, $x\approx 0$ if and only if $x_i\approx 0,~\forall~i=1,2,\dots,n$.
We say that  $x=(x_1,x_2,\dots,x_n)$, and $y=(y_1,y_2,\dots,y_n)\in\overline{\mathbb{K}}^n$
 are associated, denoted by $x\approx y$, if and only if $x_i$ and
 $y_i$ is associated in $\overline{\mathbb{K}}$ for all $1\le i\le
 n$. If there exists $a=(a_1,a_2,\dots,a_n)\in\mathbb{K}^n$ with
 $a\approx x$, i.e., $a_i\approx x_i,~\forall~i=1,2,\dots, n$, then
 $x\in\overline{\mathbb{K}}^n$ is said to be associated with $a$ and
 $a$ is so-called the shadow of $x$.

 Using the definition of the set
 $\overline{\mathbb{K}}_{s_{as}}:=\{z\in\overline{\mathbb{K}}_s:\exists~a\in\mathbb{K}~\mbox{such
   that}~z\approx a\}$ defined in the Section 2 of \cite{JO}, we can define the following 
 subset of $\overline{\mathbb{K}}$ 
\[\overline{\mathbb{K}}_{as}:=\{z\in\overline{\mathbb{K}}:
\exists~a\in\mathbb{K}~\mbox{such
   that}~z\approx a\},\] 
so-called the shadow of $\overline{\mathbb{K}}$,
see (\cite{JRJ}, Definition 1.3). Note that this is the set of all
 elements $y\in\overline{\mathbb{K}}$ such that there exists
 $p\in\mathbb{N}$ with \[\lim_{\varepsilon\downarrow
   0}\hat{y}(\varphi_{\varepsilon})=0,~\forall~\varphi\in\mathcal{A}_p(\mathbb{K})\]
 exists for some, and hence all representative $\hat{y}$ of $y$. It is
 easy to see that $\overline{\mathbb{K}}_{as}$ is in fact a subalgebra
 of $\overline{\mathbb{K}}$. Let
 $\alpha:\overline{\mathbb{K}}_{as}\to\mathbb{K}$ be defined by
 $y\mapsto\alpha(y):=\lim\limits_{\varepsilon\downarrow
   0}\hat{y}(\varphi_\varepsilon)$ for some
 $p\in\mathbb{N},~\forall~\varphi\in\mathcal{A}_p(\mathbb{K})$. 
Then it is easy to see that $\alpha$ is a $\mathbb{K}$-algebra 
surjective homomorphism. We shall denote by $\overline{\mathbb{K}}_0$ 
its kernel which is the subring of $\overline{\mathbb{K}}$ of the 
elements associated to zero, i.e., for all $\hat{y}$ representative of
$y$ there exists $p\in\mathbb{N}$, such that \[\lim_{\varepsilon\downarrow
   0}\hat{y}(\varphi_\varepsilon)=0,~\forall~\varphi\in\mathcal{A}_p(\mathbb{K}),\]
see (\cite{JRJ}, Definition 1.3). Thus we extend the (\cite{JO}, Proposition
 2.15) in the next result which is an immediate consequence of Lemma
 \ref{lema2.1}, item $ii)$ in the end of Section \ref{sec-init}.

 \begin{prop}\label{homom-1}
 \begin{enumerate}
 \item[$a)$] The algebra homomorphism $\alpha$ defined above is
 continuous for the induced topology on $\overline{\mathbb{K}}_{as}$.
 \item[$b)$] $\overline{\mathbb{K}}_{as}$ and $\overline{\mathbb{K}}_0$
   are open subalgebras of $\overline{\mathbb{K}}$ containing $B_1$.
 \end{enumerate}
 \end{prop}    

 We finish this section with the following:

 \begin{prop}\label{include}
 The following assertions hold.
 \begin{enumerate}
 \item[$i)$] $B_1(x)\subset\tilde{\Omega}_c\subset
   B_1'(0),~\forall~x\in\tilde{\Omega}_c$;
 \item[$ii)$] $\tilde{\Omega}_c$ is an open subset of
   $\overline{\mathbb{K}}^n$;
 \item[$iii)$] If $x_0\in\Omega$ and $V=\Omega\setminus\{x_0\}$, then
   $\tilde{V}_c\subset\tilde{\Omega}_c\setminus B_1(x_0)\subset S_1(x_0)$.
 \end{enumerate}
 \end{prop}

 \begin{proof}
 $i)$ Let $x\in\tilde{\Omega}_c$ and take $y\in\overline{\mathbb{K}}^n$
 such that $\Vert y-x\Vert<1$. Then $(y-x)\in B_1$ and from Proposition \ref{homom-1} item $b)$  
 we have that $(x-y)\approx 0$, i.e., there exists $p\in\mathbb{N}$
 such that \[\lim_{\varepsilon\downarrow
   0}(\hat{x}-\hat{y})(\varphi_\varepsilon)=0,~\forall~\varphi\in\mathcal{A}_p(\mathbb{K}).\]
 Thus, if  $K\subset\subset\Omega$ is such that $x_{\varphi_\varepsilon}\in
 K$, then there exists a compact subset $L$ such that $K\subset L$ and
 for $\varepsilon$ sufficiently small. Thus, $y_{\varphi_\varepsilon}\in L$
 and we have that  $y\in\tilde{\Omega}_c$. Note that $(x_{\varphi_\varepsilon})$
 is a bounded sequence and we get that $\Vert x\Vert\le 1$. We easily obtain that $ii)$ from $i)$.
 For $iii)$ we take $x\in\tilde{V}_c$ and let $K\subset\subset
 \Omega\setminus\{x_0\}$ with $x_{\varphi_\varepsilon}\in K$ for
 $\varepsilon$ sufficiently small. Then there exists
 $r\in\mathbb{R}_+^*$ such that $\vert x-x_0\vert>r$ for $\varepsilon$
 sufficiently small. Hence, $\Vert x-x_0\Vert\le 1$, i.e.,
 $x\in\tilde{\Omega}_c\setminus B_1(x_0)$. From $i)$, we obtain that
 $\tilde{\Omega}_c\setminus B_1(x_0)\subset S_1(x_0)$. 
 \end{proof}

\section{Colombeau's full diferential algebra 
$\kappa(\mathcal{G}(\Omega))$}\label{sec-4}

 In this section, we show that the Fundamental  
 Theorem of Calculus holds in this framework. 
 Moreover, we show that the embedding theorem 
 and the open mapping theorem hold. Indeed, we 
 shall extend the main results  presented in \cite{OJR} 
 that contains a proposal of a new differential 
 calculus.

 From now on  Colombeau's full generalized functions, 
 $\mathcal{G}(\Omega)$, will be endowed with the sharp
 topology, see \cite{JFO,AGJ} and \cite{JJO}. Moreover, let 
 $\mathscr{C}^1(\tilde{\Omega}_c,\overline{\mathbb{K}})$ with the
 topology of pointwise convergence, i.e., $f_n\underset{n\to\infty}{\to} f$ 
 if and only if $f_n(x)\underset{n\to\infty}{\to} f(x),~\forall~x\in\tilde{\Omega}_c$. 
 
 Now, we state the main result of this section which is a
 generalization of (\cite{OJR}, Theorem 4.1)

 \begin{thm}[Embedding Theorem]\label{embedding} 
Let $\Omega$ be an open subset of $\mathbb{R}^n$. The function
   $\kappa:\mathcal{G}(\Omega)\to\mathscr{C}^1(\tilde{\Omega}_c,
 \overline{\mathbb{K}})$ defined by $\kappa(f)(x)=f(x)$ 
   is an injective homomorphism of
   $\overline{\mathbb{K}}$-algebras. Moreover, $\kappa$ is continuous
   and $$\kappa\left(\frac{\partial f}{\partial
       x_i}\right)=\frac{\partial (\kappa(f))}{\partial
     x_i},~\forall~f\in\mathcal{G}(\Omega)~\mbox{and}~1\le i\le n.$$
 \end{thm}

 \begin{proof}
 We claim that $\kappa$ is a homomorphism of $\overline{\mathbb{K}}$-algebras. In fact,
 if $f,g\in\mathcal{G}(\Omega)$ and $a\in\overline{\mathbb{K}}$, then
 $\kappa(af)(x)=(af)(x)=af(x)=a\kappa(f)(x),~\forall~x\in\tilde{\Omega}_c$
 and we have that  
 $\kappa(af)=a\kappa(f)$.  For each $f,g\in  G(\Omega)$ we have that 
 $\kappa(f+g)(x)=(f+g)(x)=f(x)+g(x)=\kappa(f)(x)+\kappa(g)(x)=(\kappa(f)+\kappa(g))(x),~\forall~x\in\tilde{\Omega}_c$. 
 and it follows that
 $\kappa(fg)(x)=(fg)(x)=f(x)g(x)=\kappa(f)(x)\kappa(g)(x)=(\kappa(f)\kappa(g))(x),~\forall~x\in\tilde{\Omega}_c$. 
 Note that the injectivity of $\kappa$ is immediately obtained from  Theorem
 \ref{caracter}. Since $\kappa$ is a homomorphism, then  we only need to prove the 
 continuity at zero to get the continuity of $\kappa$.  Let 
 $f_n\underset{n\to\infty}{\rightarrow} 0$ in
 $\mathcal{G}(\Omega),~x\in\tilde{\Omega}_c$ and
 $\hat{f}_n,~(x(\varphi))_\varphi$ be representatives of these
 elements. Choose an exhaution $(\Omega_m)_{m\in\mathbb{N}}$ of $\Omega$
   and fix $m_0$ such that there exists $p\in\mathbb{N}$ with
   $x(\varphi_\varepsilon)\in\Omega_{m_0},~\forall~\varphi\in\mathcal{A}_p(\mathbb{K}),~\forall~\varepsilon\in
   I_\eta$. Then,
   $\kappa(f_n)(x)=f_n(x)\underset{n\to\infty}{\rightarrow}
   0$ which implies that  $\kappa(f_n)\underset{n\to\infty}{\rightarrow} 0$ and,
    we have that  $\kappa$ is continuous at zero. It remains to show
    that $\kappa(f)$ is differentiable at $x_0\in\tilde{\Omega}_c$ and 
   prove this fact we firstly suppose that $n=1$. Let $x_0\in\tilde{\Omega}_c$
   whose support is contained in a compact subset to $K\subset\Omega$. We claim that 
    $\kappa(f)$ is differentiable at $x_0$ and 
   $D(\kappa(f))(x_0)=\kappa(f')(x_0)=f'(x_0)$ and this is equivalent to show that
\begin{equation}\label{eq:defin-1}
\lim_{x\to
   x_0}\frac{\kappa(f)(x)-\kappa(f)(x_0)-D(\kappa(f))(x_0)(x-x_0)}{\dot{\beta}_{\Vert x-x_0\Vert}}=0,
\end{equation}
  where  $x\in\tilde{\Omega}_c$ with $\Vert x-x_0\Vert<1$. In fact,
\begin{eqnarray}\label{eq:defin-2}
 D(\kappa(f))(x_0)&=&\lim_{n\to\infty}\frac{\kappa(f)(x_0+\dot{\alpha}_n)-\kappa(f)(x_0)}{\dot{\alpha}_n}\nonumber\\
 &=&\lim_{n\to\infty}\frac{f(x_0+\dot{\alpha}_n)-f(x_0)}{\dot{\alpha}_n}\nonumber\\
 &=&f'(x_0)=\kappa(f')(x_0),
\end{eqnarray}
 since  
 $\kappa(f)=f\in\mathscr{C}^1(\tilde{\Omega}_c;\overline{\mathbb{K}})$. 
 Replacing (\ref{eq:defin-2}) in (\ref{eq:defin-1}), and by the 
 definition of $\kappa$, we have that
 (\ref{eq:defin-1}) is equivalently to  
 \begin{equation}\label{eq:defin-3}
 \lim_{x\to
   x_0}\frac{f(x)-f(x_0)-\kappa(f')(x_0)(x-x_0)}{\dot{\beta}_{\Vert x-x_0\Vert}}=0,
 \end{equation}
 and again by the Remark \ref{deriv-3}-$b)$ the equality in
 (\ref{eq:defin-3}) is true for $\kappa(f)=f$. Thus, we  have 
  the existence of partial derivatives. Now, to prove the differentiability
 in the general case it is enough to repeat, for
 $\overline{\mathbb{K}}^n$ endowed with the product topology.
 \end{proof}

 Let $J$ be an open interval of $\mathbb{R}, ~f\in\mathcal{G}(J)$ and
 $a,b\in\tilde{J}_c$. We define 
 \begin{equation}\label{eq:integ-1}
 \left(\intop_a^b\kappa(f)\right)(\varphi):=\intop_{a(\varphi)}^{b(\varphi)}f(\varphi,t)\dd
 t,
 \end{equation} 
 where $(a(\varphi)),(b(\varphi))$ and $f(\varphi,\cdot)$ are
 representatives of $a,b$ and $f$, respectively,  the second
 integral is Riemann integral and 
 $\intop_a^b\kappa(f)$ is a well-defined element of
 $\overline{\mathbb{R}}$. It is easy to see that 
 \begin{enumerate}
 \item[$i)$] If $f,g\in\mathcal{G}(J),~\lambda\in\overline{\mathbb{R}}$
   and $a,b,c\in\tilde{J}_c,~a\le c\le b$,
   then \[\intop_a^b\kappa(f+\lambda
   g)=\intop_a^b\kappa(f)+\lambda\intop_a^b\kappa(g)
 \quad \mbox{and}\quad \intop_a^b\kappa(f)=\intop_a^c\kappa(f)+\intop_c^b\kappa(f);\]
 \item[$ii)$] If $a,b\in J$, then $\intop_a^b\kappa(f)=\intop_a^bf$,
   where the second integral is the integral of generalized functions
   (see, for example, \cite{MK}). 
 \end{enumerate}

 We show by using Theorem \ref{embedding}, that the Fundamental
 Theorem of Calculus  holds in this framework, i.e., if $J$ is an open
 interval of $\mathbb{R},~a\in\tilde{J}_c,~f\in\mathcal{G}(J)$ and
 $F_x$ is the function defined on $\tilde{J}_c$ by
 $F_x=\intop_a^x\kappa(f),$ then $F_x$ is a differentiable function
 and \[F_x'=D\left(\intop_a^x\kappa(f)\right)=\kappa(f),\] that is,
 $F_x'=\kappa(f)$. In fact, let $d\in J$, $(a(\varphi))$ a
 $(f(\varphi,\cdot))$ be representatives of $a$ and $f$,
 respectively. Then
 \begin{eqnarray}
 F_x(\varphi)&=&\left(\intop_a^x\kappa(f)\right)(\varphi)\nonumber\\
 &=&\intop_{a(\varphi)}^{x(\varphi)}f(\varphi,t)\dd t\nonumber\\
 &=&\intop_{a(\varphi)}^df(\varphi,t)\dd t+\intop_d^yf(\varphi,t)\dd
     t\nonumber\\
 &=&\intop_{a(\varphi)}^df(\varphi,t)\dd t+\kappa(G_y),
 \end{eqnarray} 
 where $G_y:=\intop_d^y(f(\varphi,t)\dd t$ in
 $\mathcal{G}(J)$ and $y=x(\varphi)$. Thus, \[F_x'=\left(\intop_{a(\varphi)}^df(\varphi,t)\dd
   t+\kappa(G_y)\right)'=(\kappa(G_y))'=\kappa(G_y')=\kappa(f),\]
 because $G_y'=\left(\intop_d^yf(\varphi,t)\dd t\right)'=f(\varphi,y)$
 a representative of $f\in\mathcal{G}(J)$. These results show the
 consistency of our proposal and allow one to use standard techniques
 of differential calculus.

 \begin{defn}\label{norma-1}
 Let $x\in\overline{\mathbb{K}}$ and  $\hat{x}$  one of its
 representatives. Then the function
 $|\hat{x}|:\mathcal{A}_0(\mathbb{K})\to\mathbb{R}_+$ defined  by
 $|\hat{x}|(\varphi)=|\hat{x}(\varphi)|$  rises to an element
 $|x|\in\overline{\mathbb{R}}_+$ which depends only on $x$ and it is
 called the module of $x$.
 \end{defn}

 The next lemma was called in \cite{OJR} as Generalized
 Cauchy-Schwarz inequality in the case of the simplified algebra. The
 same holds for the case of the full algebras as the next result 
 shows whose proof is analogous to the classical one. 

 \begin{lem}\label{norma-2}
 Let $x,y\in\overline{\mathbb{K}}^n$. Then $|\langle x|y\rangle|\le
 [x]_2[y]_2$, where $[\cdot]_2:=\left(\sum_{i=1}^n|\cdot_i|^2\right)^{\frac{1}{2}}$. 
 \end{lem}

 \begin{prop}\label{norma-3}
 Let $\Omega$ be open subset of $\mathbb{R}^n$ and
 $f\in\mathcal{G}(\Omega)$. The following assertions hold:
 \begin{enumerate}
 \item[$1)$]
   $\kappa(\mathcal{G}(\Omega))\subset\mathscr{C}(\tilde{\Omega}_c;\overline{\mathbb{K}})$
   and
   $\kappa(\mathcal{G}(\Omega))\ne\mathscr{C}(\tilde{\Omega}_c;\overline{\mathbb{K}})$;
 \item[$2)$] If $\Omega$ is connected, then for each
   $x,y\in\tilde{\Omega}_c$ there is $c\in\tilde{\Omega}_c$ such
   that \[\kappa(f)(x)-\kappa(f)(y)=\langle\nabla\kappa(f)(c)|x-y\rangle\]
   and \[|\kappa(f)(x)-\kappa(f)(y)|\le
   [\nabla\kappa(f)(c)]_2[x-y]_2;\]
 \item[$3)$] If $\omega$ is convex and $D(\kappa(f))=0$, then $f$ is a
   constant;
 \item[$4)$] If $K\subset\Omega$ is a compact subset, then
   $\kappa(f)\Big|_{\tilde{K}}$ is bounded. Moreover, if
   $0\in\Omega$, then $\kappa(f)\Big|_{\tilde{K}}\cap B_1(0)$ is
     Lipschitz function, where
     \[\tilde{K}:=\{x\in\tilde{\Omega}|~\exists ~\mbox{a
       representative}~(x(\varphi)) ~\mbox{of}~x~\mbox{such
       that}~x(\varphi)\in K,~\forall~\varphi\in\mathcal{A}_0(\mathbb{K})\};\]
 \item[$5)$] If $n=1,~m\in\mathbb{N}^*$ and $\Omega$ is convex, then
   given $x,y\in\tilde{\Omega}_c$ there is $z\in\tilde{\Omega}_c$ such
   that \[\kappa(f)(x)=\sum_{0\le j\le
     m}\frac{\kappa(f^{(j)})(y)(x-y)^j}{j!}+\frac{\kappa(f^{m+1})(z)(x-y)^{m+1}}{(m+1)!}\]
 and there exist $(x(\varphi)),(y(\varphi))$ and $(z(\varphi))$,
 representatives of $x,y$ and $z$, respectively with $x(\varphi)\le
 z(\varphi)\le y(\varphi)$;
 \item[$6)$] If $n=1,~m\in\mathbb{N}^*$ and $x\in\tilde{\Omega}_c$,
   then \[\lim_{\Vert x-y\Vert \rightarrow 0}\frac{1}{(\dot{\beta}_{\Vert x-y\Vert})^n}\left(\kappa(f)(x)-\sum_{0\le j\le
     m}\frac{\kappa(f^{(j)})(y)(x-y)^j}{j!}\right)=0.\]
 \end{enumerate}
 \end{prop}

 \begin{proof}
 The proof  is a simple adaptation of (\cite{OJR}, Proposition 4.4). From 
 Theorem \ref{embedding} together with Example \ref{deriv-5} we get
 $1$. The second part of $2)$ follows from Lemma
 \ref{norma-2}. Assertion $3)$ follows from $2)$. The other assertions
 are proved with similar arguments used in the  proof of Theorem \ref{embedding}.
 \end{proof}

 \begin{prop}\label{norma-4}
 Let $f\in\mathscr{C}^\infty(\Omega)\subset\mathcal{G}(\Omega),
 ~U:=\kappa(f)(\tilde{\Omega}_c),~V:=f(\Omega)$
 and $$\tilde{V}_c:=\{x\in\overline{\mathbb{K}}:\exists~\mbox{repres.}~(x_\varphi)_\varphi,
 ~\exists~p\in\mathbb{N}~\mbox{s.t.}~x_{\varphi_\varepsilon}\in
 K,~\forall~\varphi\in\mathcal{A}(\mathbb{K}),~\forall~\varepsilon\in
 I_\eta\}.$$ Then the following statements hold.
 \begin{enumerate}
 \item[$a)$] $U\subset\tilde{V}_c$ and $\kappa(f)$ is a bounded
   function;
 \item[$b)$] If $f$ is an open mapping, then $U=\tilde{V}_c$ and $U$ is
   an open subset of $\tilde{K}_c$
 \end{enumerate}
 \end{prop}

 \begin{proof}
 Again, we have a little adaptation of the original proof of the
 corresponding result presented in \cite{OJR}.

 $a)$ It is immediate that $U\subset\tilde{V}_c$ and
 hence, from Proposition \ref{include} item $i), ~\kappa(f)$ is a
 bounded function. 

 $b)$ Let 
 $z\in\tilde{V}_c,~(z_\varphi)_\varphi$ a representative of $z$ and
 $K\subset\subset V$ such that $z_{\varphi_\varepsilon}\in K$ for
 $\varepsilon$ sufficiently small and
 $\forall~\varphi\in\mathcal{A}_p(\mathbb{K})$. Then  there exists an $L\subset\subset\Omega$ such that $K\subset
 f(L)$, since $f$ is an open mapping.  Hence there exists an element $x\in\tilde{\Omega}_c$, whose
 support is contained in $L$, such that
 $f(x_{\varphi_\varepsilon})=z_{\varphi_\varepsilon}$. From Proposition
 \ref{include} item $ii), ~U$ is open.  
 \end{proof}

 Notice that we do not really need that $f$ is an open mapping, what we
 actually need is that there exists an exhaustion
 $(\Omega_n)_{n\in\mathbb{N}}$ of relatively compact sets of $\Omega$
 such that $(f(\Omega_n))_{n\in\mathbb{N}}$ is an exhaustion of $\Ima(f)$.

 The following corollary is called the Open Mapping Theorem.

 \begin{cor}[Open Mapping Theorem]\label{norma-5}
 Let $f\in\mathscr{C}^\infty(\Omega)$ be an open mapping. Then for every
 open subset $W\subset\Omega$ we have that $\kappa(f)(\tilde{W}_c)$ is open.
 \end{cor}

 \begin{prop}\label{norma-6}
 Let $\Omega$ be connected, $f\in\mathcal{G}(\Omega)$ and suppose that
 $\Ima(\kappa(f))$ is a discret set. Then $f$ is constant.
 \end{prop}

 \begin{proof}
 The prove is same that appears in (\cite{OJR} for Proposition 4.7). 
 \end{proof}

\section{Holomorphic and analytic generalized 
functions and applications}\label{sec-5}

In this section we shall define the notions of holomorphic 
and analytic functions in the framework of Colombeau's full
algebra. Here $\mathbb{K}$ shall always stand for $\mathbb{C},~\Omega$
denotes a non-void open set of
$\mathbb{C},~\mathscr{H}(\Omega)\coloneqq\{f\in\mathscr{C}^1(\Omega;\mathbb{C}):\bar{\partial}
   f=0\}$ and
   $\mathscr{H}\mathcal{G}(\Omega)\coloneqq\{f\in\mathcal{G}(\Omega;\mathbb{C}):\bar{\partial}f=0\}$. It
   is obvious that
   $\overline{\mathbb{C}}=\overline{\mathbb{R}}+i\overline{\mathbb{R}}$,
   where $i^2=-1$. Thus, we consider $\overline{\mathbb{C}}$ to be
   $\overline{\mathbb{R}}$-isomorphic to $\overline{\mathbb{R}}^2$. If
   $z=x+iy$, with $x,y\in\overline{\mathbb{R}}$, then we define the
   following operators: \[\frac{\partial}{\partial z}\coloneqq
   \frac{1}{2}(\partial_x-i\partial_y),\qquad\mbox{and}\qquad
\frac{\partial}{\partial\bar{z}}\coloneqq\frac{1}{2}(\partial_x+i\partial_y).\]

As in \cite{OJR} our Embedding Theorem  gives, in
the obvious way, an Embedding Theorem in the Complex case. We also
have an Open Mapping Theorem in this case.  

\begin{thm}\label{holom-1}
If $f\in\mathscr{H}$ is non-constant, then $(\kappa(f))(\tilde{W}_c)$
is an open subset for all sets $W\subset\Omega$.
\end{thm}

\begin{proof}
This follows at once from Proposition \ref{norma-4}.
\end{proof}

\begin{defn}\label{holom-2}
 Let $f\in\mathcal{G}(\Omega)$. We shall say that $f$ is sub-linear in
 $\Omega$ if there exists a representative $\hat{f}$ of $f$ with the
 following property: for all $x\in\tilde{\Omega}_c$, there exists a
 representative $(x(\varphi))$ of $x,~k\in\mathbb{R}$, a sequence
 $(\eta_n)_{n\in\mathbb{N}}\in I=]0,1]$ and sequences
 $(C_n)_{n\in\mathbb{N}},~(p_n)_{n\in\mathbb{N}}$ in $\mathbb{R}$ such
 that $\lim\limits_{n\to\infty}(p_n+k_n)=\infty$
 and \[|\hat{f}^{(n)}(\varphi_\varepsilon,x(\varphi_\varepsilon))|\le
 C_n\varepsilon^{-p_n},~\forall~\varphi\in\mathcal{A}_{p_n}(\mathbb{K}),~\forall
 \varepsilon\in I_{\eta_n}=]0,\eta_n[.\]
\end{defn}

Note that the definition does not depend on the representative of
$f$. It is immediate to verify that the set of all sublinear functions
of $\Omega$ is a $\mathbb{K}$-subalgebra of $\mathcal{G}(\Omega)$.

\begin{exam}\label{holom-3}
If $f\in\mathscr{C}^\infty(\Omega)\subset\mathcal{G}(\Omega)$, then
$f$ is sub-linear.
\end{exam}

\begin{defn}\label{holom-4}
 Let $U\subset\overline{\mathbb{K}}$ be an open subset and $z_0\in
 U$. We say that $f:U\to\overline{\mathbb{K}}$ is analytic in $z_0$ if
 there exists a sequence 
 $(a_n)_{n\in\mathbb{N}}\in\overline{\mathbb{K}}$ and a series of the
 form $\sum\limits_{n\ge 0}a_n(z-z_0)^n$ which converges in a neighborhood of
 $z_0$  such that $f(z)=\sum\limits_{n\ge 0}a_n(z-z_0)^n$ in this
 neighborhood. Moreover, we say that $f$ is analytic if 
 $f$ is analytic for all $z_0\in U$ and we write
 $f\in\mathscr{A}\mathcal{G}(U)$.
\end{defn}

As in \cite{OJR}, in the proof of the results below we use the
 following obvious fact, which holds for general complete ultra-metric
 abelian groups $G$, but here we restrict our attention to the case
 $G=\overline{\mathbb{K}}$: If $(b_n)_{n\in\mathbb{N}}$ is a sequence
 in $\overline{\mathbb{K}}$, then $\sum\limits_nb_n$ converges if and only if 
$\lim\limits_{n\to\infty}\Vert b_n\Vert=0$.

\begin{thm}\label{holom-5}
 Let $r>0,~z_0\in\overline{\mathbb{K}}$ and $f(z)=\sum\limits_{n\ge
   0}a_n(z-z_0)^n\in\mathscr{A}\mathcal{G}(B_r(z_0))$. Then $f$ is
 differentiable and $f'(z)=\sum\limits_{n\ge
   1}na_n(z-z_0)^{n-1}$.
\end{thm}

\begin{proof}
 Let $r>0,~z\in B_r(z_0)$ and $s\in\mathbb{R}_+^*$ such that $\Vert
 z-z_0\Vert<s<r$. Since $\sum\limits_{n\ge
   0}a_n(\dot{\beta}_s)^n$, where $\dot{\beta}_s=\dot{\alpha}_{-\log(s)}$,
 converges we have, for $w\in B_s(z_0)\subset B_r(z_0)$, that
\begin{eqnarray*}
 \lim_{n\to\infty}\Vert na_n(w-z_0)^{n-1}\Vert&=&\lim_{n\to\infty}\Vert
                                                  a_n(w-z_0)^{n-1}\Vert\\
 &\le&\lim_{n\to\infty}\Vert a_n\Vert\Vert (w-z_0)^{n-1}\Vert\\
 &\le&\lim_{n\to\infty}\Vert a_n\Vert s^{n-1}\\
 &\le&s^{-1}\lim_{n\to\infty}\Vert a_n(\dot{\beta}_s)^n\Vert\\
 &=&0
 \end{eqnarray*}
 which implies that $\sum\limits_{n\ge 1}na_n(w-z_0)^{n-1}$ converges
 uniformly on $B_s(z_0)$. Hence,
 if \[g(w)\coloneqq\frac{f(w)-f(z)-(w-z)\sum\limits_{n\ge
     1}na_n(z-z_0)^{n-1}}{\dot{\beta}_{\Vert w-z\Vert}},~\forall~w\in
   B_s(z_0)\setminus\{z\},\] then
 \begin{eqnarray*}
 \lim_{w\to z}g(w)&=&\lim_{w\to z}\lim_{m\to\infty}\sum\limits_{1\le
                      n\le
                      m}a_n\left(\frac{(w-z_0)^n-(z-z_0)^n-n(z-z_0)^{n-1}(w-z)}{\dot{\beta}_{\Vert
                      w-z\Vert}}\right)\\
 &=&\lim_{m\to\infty}\lim_{w\to z}\sum\limits_{1\le
                      n\le
                      m}a_n\left(\frac{(w-z_0)^n-(z-z_0)^n-n(z-z_0)^{n-1}(w-z)}{\dot{\beta}_{\Vert
                      w-z\Vert}}\right)\\
 &=&\lim_{m\to\infty}\sum\limits_{1\le n\le m}0\\
 &=&0.
\end{eqnarray*} 
  Thus $f'(z)=\sum\limits_{n\ge 1}na_n(z-z_0)^{n-1}$.  
\end{proof}
 
It is convenient to point out that in the last result we used the
following fact: the sequence $(\psi_m)_{m\in \mathbb{N}}$, where 
$\psi_m(w):=\sum_{1\leq n \leq m}
a_n(\frac{(w-z_0)^n-(z-z_0)^n-n(z-z_0)^{n-1}(w-z)}{\dot{\beta}_{||w-z||}}$ 
uniformly converges on $B_s(z_0)\setminus\{z\}$ to $g$, and  the change of the order of
 the limits in the second equality follows from a classical result presented in 
(\cite{RW}, Theorem 7.11), which clearly holds for function with
 values in $\overline{\mathbb{K}}$, because it is complete. 

\begin{cor}\label{holom-6}
 Let $r>0,~z_0\in\overline{\mathbb{K}}$ and $f(z)=\sum\limits_{n\ge
   0}a_n(z-z_0)^n$ for $z\in B_r(z_0)$. Then
 $f\in\mathscr{C}^\infty(B_r(z_0);\overline{\mathbb{K}})$ and for
 $k\in\mathbb{N}^*$ one has $f^{(k)}(z)=\sum\limits_{k\ge
   0}n(n-1)\dots(n-k+1)a_n(z-z_0)^{n-k}$. In particular, $k!a_k=f^{(k)}(z_0)$.
\end{cor}

\begin{thm}\label{holom-7}
 Let $f\in\mathcal{G}(\Omega)$. The following assertions hold.
 \begin{enumerate}
 \item[$i)$] If $\kappa(f)$ is analytic, then $f$ is sub-linear;
 \item[$ii)$] If $f\in\mathscr{H}\mathcal{G}(\Omega)$ and $f$ is
   sub-linear, then $\kappa(f)$ is analytic and for all
   $z_0\in\tilde{\Omega}_c$ there exists $r\in]0,1[$ such that for all
   $z\in B_r(z_0)$ one has $\kappa(f)(z)=\sum\limits_{n\ge
     0}\frac{\kappa(f)^{(n)}(z_0)}{n!}(z-z_0)^n$. Moreover,  this series
   converges uniformly in $B_r(z_0)$ and $\frac{\partial}{\partial\bar{z}}(\kappa(f))=0$.
\end{enumerate}
\end{thm}

\begin{proof}
 $i)$ Let $z_0\in\tilde{\Omega}_c$ and suppose that
 $\kappa(f)(z)=\sum\limits_{n\ge 0}a_n(z-z_0)^n,~\forall~z\in B_R(z_0)$
 with $R\in\mathbb{R}_+^*$, where  by Corollary \ref{holom-6},  
we have that $n!a_n=\kappa(f)^{(n)}(z_0)$. Let $r\in]0,1[$ be such that
 $\Vert\dot{\alpha}_r\Vert=e^{-r}<R$ and define
 $z=z_0+\dot{\alpha}_r$. Since $\Vert
 z-z_0\Vert=\Vert\dot{\alpha}_r\Vert=e^{-r}<R$ one has that the
 series \[\sum\limits_{n\ge 0}a_n(z-z_0)^n=\sum\limits_{n\ge
   0}(\dot{\alpha}_r)^n\] converges and from Corollary \ref{norm} item
 $iv)$, we have that \[0=\lim_{n\to\infty}\Vert
 a_n(\dot{\alpha}_r)^n\Vert=\lim_{n\to\infty}e^{-(V(a_n)+nr)}.\]
 Let $(f_\varphi(\cdot))_\varphi$ be a representative of $f$ and
 $(z_{0\varphi})_\varphi$ one of $z_0$. Take $k=r,~c_n=1$ and
 \begin{equation*}
 p_n=\left\{\begin{array}{l}
 V(a_n)-1,~\mbox{if}~V(a_n)\in\mathbb{R};\\
 n, ~\mbox{if}~V(a_n)=\infty ~(\mbox{i.e.},~a_n\in\mathcal{N}(\mathbb{K})).
 \end{array}\right.
 \end{equation*} 
 Then $\lim\limits_{n\to\infty}(p_n+nk)=\infty$ and it follows that \[\vert
 f_{\varphi_\varepsilon}^{(n)}(z_{0\varphi_\varepsilon})\vert=n!\vert
 a_n\vert\le c_n\varepsilon^{p_n}\] for $\varepsilon$ sufficiently
 small.

 $ii)$ Let $(K_\nu)_\nu$ be an exhaustive sequence of compact subsets of
 $\Omega$ such that $\mathring{K}_\nu$ is a
 $\mathscr{C}^\infty$-strictly pseudoconvex domain for each
 $\nu\in\mathbb{N}$, see (\cite{HL}, Corollaries 1.5.6 and 
 1.5.11). From (\cite{JA1}, Theorem 2)  there exists a representative
 $(f_{\nu\varphi}(\cdot))_\varphi$ of $f$ such that
 $f_{\nu\varphi_\varepsilon}(\cdot)\in\mathscr{H}(\mathring{K}_\nu)$
 for $\varepsilon\in I$ and $\nu\in\mathbb{N}$. Take
 $z_0\in\tilde{\Omega}_c$. Since $f$ is sublinear then there exists
 $(\eta_n)_{n\in\mathbb{N}}$ a sequence in $I, k\in\mathbb{R},
 (z_{0\varphi})_\varphi$ a representative of $z_0$, and the sequences
 in $\mathbb{R}: (c_n)_{n\in\mathbb{N}},
 ~\mbox{and}~(p_n)_{n\in\mathbb{N}}$, such that
 $\lim\limits_{n\to\infty}(p_n+k_n)=\infty$ with \[\vert
 f_{\varphi_\varepsilon}^{(n)}(z_{0\varphi_\varepsilon})\vert\le
 c_n\varepsilon^{p_n},~\forall~\varepsilon\in I_{\eta_n},
 n\in\mathbb{N}.\] Thus, if $0<r<\Vert\dot{\alpha}_{\vert k\vert}\Vert=e^{-\vert
   k\vert}$, then  \[\lim_{n\to\infty}\left\Vert\frac{\kappa(f)^{(n)}}{n!}(z-z_0)^n\right
 \Vert\le\lim_{n\to\infty}e^{-(p_n+k_n)}=0,~\forall~z\in B_r(z_0).\]
 Hence, \[\sum\limits_{n\ge
   0}\frac{\kappa(f)^{(n)}}{n!}(z-z_0)^n\] converges uniformly in
 $B_r(z_0)$. Let $\nu\in\mathbb{N}$ and $s\in\mathbb{R}_+^*$ such that
 $z_{0\varphi_\varepsilon}\in K_\nu\subset \mathring{K}_{\nu+1}$ and
 define \[\bar{D}_s(z_{0\varphi_\varepsilon})\coloneqq\{\lambda\in\mathbb{C}:
 \vert\lambda-z_{0\varphi_\varepsilon}\vert\le
 s\}\subset\mathring{K}_{\nu+1},~\forall~\varepsilon\in I.\]
 By the fact all  the elements of $B_1$ are associated to zero
 we have that  \[f_{\nu\varphi_\varepsilon}(x)=\sum\limits_{n\ge
   0}\frac{f_{\nu\varphi_\varepsilon}^{(n)}(z_{0\varphi_\varepsilon})}{n!}(x-z_{0\varphi_\varepsilon})^n,~\forall
 ~x\in\bar{D}(z_{0\varphi_\varepsilon})\] and $\varepsilon$
 sufficiently small. Consequently, we conclude
 that \[\kappa(f)(z)=\sum\limits_{n\ge
   0}\frac{\kappa(f)^{(n)}(z_0)}{n!}(z-z_0)^n,~\forall~z\in B_r(z_0).\]
 Moreover, as \[\frac{\partial}{\partial\bar{z}}(\kappa(f))=\kappa\left(\frac{\partial}{\partial\bar{z}}f\right)=0,\]
 we have that $\frac{\partial}{\partial\bar{z}}(\kappa(f))=0$.
 \end{proof}

Note that in the proof of the above theorem we actually obtain a lower
bound for the radius of convergence in each point.
 
We finish this section with the following result that 
is an easy consequence of the last theorem. 

\begin{cor}\label{holom-8}
Let $f\in\mathscr{H}\mathcal{G}(\Omega)$. Then $\kappa(f)$ is analytic
if and only if  $f$ is sub-linear. 
\end{cor}

\subsection{Some applications}
 In \cite{AS} the authors used (\cite{JO}, Proposition $2.5$) 
 to show  the non existence of a solution for a
 certain first-order partial differential equation. In this article 
 the following question is still raised: The result of no solution 
 existence for such an equation can be generalized to any linear
 operator with constant coefficients? In \cite{OJR} generalizes and
 gives an answer to this question, making use of the Open mapping 
 Theorem in a simple, but interesting way into
 $\overline{\mathbb{K}}_s$. In this paper we extend this result to 
 $\overline{\mathbb{K}}$ using the results of simplified generalized 
 holomorphic functions of \cite{OJR}, that have been extended here for 
 the case of Colombeau's full algebras.

 \begin{thm}\label{aplica-1}
 Let $\Omega$ be a connected open subset of
 $\mathbb{C}^n,~f\in\mathscr{H}(\Omega)$ non-constant,
 $L=\sum\limits_{1\le k\le m}a_k\frac{\partial}{\partial z_k}$ a linear
 differential operator with constant coefficients $a_1,a_2,\dots,a_m$
 belonging to $\overline{\mathbb{K}}$ and $\mathfrak{J}$ the ideal
 generated by $\{a_1,a_2,\dots, a_m\}$. If there exists
 $u\in\mathcal{G}(\Omega)$ such that $L(u)=f$, then
 $\mathfrak{J}=\overline{\mathbb{K}}$.
 \end{thm}

 \begin{proof}
 Suppose $\mathfrak{J}$ is an ideal of $\overline{\mathbb{K}}$. If
 there exists $u\in\mathcal{G}(\Omega)$ such that $L(u)=f $, then
 $\Ima(\kappa(f))\subset\mathfrak{J}$. By 
 (\cite{JRJ}, Proposition 2.6 item $i)$), we have that $\mathfrak{J}$ is a rare subset of 
 $\overline{\mathbb{K}}$. So, $\Ima(\kappa(f))$ would not 
 be open which contradicts the Theorem 
 \ref{holom-1}.
 \end{proof}

 \begin{prop}
 Let $A\in\mathcal{S}_f$ and $L$ the differential operator defined by
 $L=(\chi_AD^2+\chi_{A^c}\id)$. Then the solutions of $L(u)=0$ are all
 of the form $\chi_Af$ with $D^2f\neq 0$ and 
 $\chi_A\lambda^2+\chi_{A^c}=0,~\forall~\lambda\in\overline{\mathbb{K}}.$
 \end{prop}

 \begin{proof}
 If $D^2f=0$, then we have that  $L(\chi_Af)=0$. Now, 
 let $f$ be such that $L(f)=0$. Since $\chi_AD^2f+\chi_{A^c}f=0$,  then we 
 have that $\chi_A(\chi_AD^2f+\chi_{A^c}f)=0$. Thus, $\chi_AD^2f=0$ 
 and $\chi_{A^c}f=0$. By the fact that  $1=\chi_A+\chi_{A^c}$ we have  that 
 $f=\chi_Af$ and $D^2f=0$. Note that if $\chi_A\lambda^2+\chi_{A^c}=0$, 
 then $\chi_{A^c}(\chi_A\lambda^2+\chi_{A^c})=0$ which implies that 
 $\chi_{A^c}=0$, this contradicts the fact that $\chi_{A^c}\ne 0$.
 So, $\chi_A\lambda^2+\chi_{A^c}\ne 0,~\forall ~\lambda\in\overline{\mathbb{K}}$. 
 \end{proof}


\end{document}